\documentclass[12pt,letterpaper,leqno]{amsart}
\usepackage{microtype}
\usepackage{amssymb}
\usepackage{amsfonts}
\usepackage{geometry}
\usepackage{hyperref}
\usepackage{enumerate}
\usepackage{color}
\usepackage{cite}

\renewcommand{\Re}{\operatorname{Re}}

\newtheorem{theorem}{Theorem}
\theoremstyle{plain}

\newtheorem{claim}[theorem]{Claim}

\newtheorem{corollary}[theorem]{Corollary}

\newtheorem{definition}[theorem]{Definition}

\newtheorem{lemma}[theorem]{Lemma}

\newtheorem{proposition}[theorem]{Proposition}
\newtheorem{remark}[theorem]{Remark}

\numberwithin{equation}{section}
\numberwithin{theorem}{section}  
\input{tcilatex}
\begin{document}
\title[Optimal $H^{1}$ $N$-body Convergence Rate via Hierarchy]{Quantitative
Derivation and Scattering of the \\
3D Cubic NLS in the Energy Space}
\author{Xuwen Chen}
\address{Department of Mathematics, University of Rochester, Rochester, NY
14627}
\email{xuwenmath@gmail.com}
\urladdr{http://www.math.rochester.edu/people/faculty/xchen84/}
\author{Justin Holmer}
\address{Department of Mathematics, Brown University, 151 Thayer Street,
Providence, RI 02912}
\email{justin$\_$holmer@brown.edu}
\urladdr{http://www.math.brown.edu/jholmer/}
\date{}

\begin{abstract}
We consider the derivation of the {defocusing cubic nonlinear Schr\"{o}%
dinger equation (NLS) on $\mathbb{R}^{3}$} from quantum $N$-body dynamics.
We reformat the hierarchy approach with Klainerman-Machedon theory and prove
a bi-scattering theorem for the NLS to obtain convergence rate estimates
under $H^{1}$ regularity. The $H^{1}$ convergence rate estimate we obtain is
almost optimal for $H^{1}$ datum, and immediately improves if we have any
extra regularity on the limiting initial one-particle state.
\end{abstract}

\maketitle

\section{Introduction}

The aim of this paper is to close, with a simple and short argument, the
regularity gap that is currently present in the literature on the derivation
of the cubic nonlinear Schr\"{o}dinger equation (NLS) from quantum many-body
dynamics on $\mathbb{R}^{3}$. Let us write the cubic NLS 
\begin{equation}
\begin{aligned} &i\partial _{t}\phi = -\Delta _{x}\phi +b_{0}\left\vert \phi
\right\vert ^{2}\phi \text{ in }\mathbb{R}^{3+1} \\ &\phi (0,x) =\phi
_{0}(x) \end{aligned}  \label{NLS:opening cubic}
\end{equation}%
and the linear $N$-body Schr\"{o}dinger equation 
\begin{equation}
i\partial _{t}\psi _{N}=H_{N}\psi _{N}\text{ in }\mathbb{R}^{3N+1}
\label{eqn:N-body}
\end{equation}%
where the $N$-body Hamiltonian is%
\begin{equation}
H_{N}=\sum_{j=1}^{N}-\Delta _{x_{j}}+\frac{1}{N}\sum_{i<j}N^{3\beta
}V(N^{\beta }(x_{i}-x_{j}))  \label{hamitonian:N-body}
\end{equation}%
and define the marginal densities $\gamma _{N}^{(k)}$ associated with $\psi
_{N}$ in kernel form by%
\begin{equation}
\gamma _{N}^{(k)}\left( t,\mathbf{x}_{k},\mathbf{x}_{k}^{\prime }\right)
=\int \psi _{N}(t,\mathbf{x}_{k},\mathbf{x}_{N-k})\bar{\psi}_{N}(t,\mathbf{x}%
_{k}^{\prime },\mathbf{x}_{N-k})d\mathbf{x}_{N-k}  \label{def:marginals}
\end{equation}%
where $\mathbf{x}_{k}=(x_{1},...,x_{k})\in \mathbb{R}^{3k}$. The main object
of study in the derivation is, then, the $N\rightarrow \infty $ limit%
\begin{equation}
\gamma _{N}^{(k)}(t)\rightarrow \left\vert \phi (t)\right\rangle
\left\langle \phi (t)\right\vert ^{\otimes k}  \label{limit:mean field limit}
\end{equation}%
in operator form\footnote{%
As usual, in the notation, we do not distinguish the kernel and the operator
it defines.}, where $\phi (t)$ is given by (\ref{NLS:opening cubic}),
provided that $\gamma _{N}^{(1)}(0)\rightarrow \left\vert \phi
_{0}\right\rangle \left\langle \phi _{0}\right\vert $. Limit (\ref%
{limit:mean field limit}) was first rigorously justified in \cite%
{E-S-Y2,E-S-Y5,E-S-Y3} assuming $H^{1}$ regularity via the now so-called
hierarchy method that concluded convergence but with no estimates on the
convergence rate as it was via a compactness argument. Later on, the work 
\cite{BOB,GM1} pioneered the study of the rate of convergence in limit (\ref%
{limit:mean field limit}) via the theory of Bogoliubov rotation /
metaplectic representations and the now so-called Fock space method, but it
requires at least $H^{4}$ regularity. We will explain more of these two
methods later in the paper, but it is obvious that there is, at the moment,
a significant regularity gap: $H^{1}$ vs $H^{4}$, between proving limit (\ref%
{limit:mean field limit}) holds and proving limit (\ref{limit:mean field
limit}) holds with a rate.\footnote{%
The gap could be less severe in 1D and 2D as the corresponding critical
regularity drops.} It is certainly of mathematical interest to reduce the
required regularity and provide an optimal result. At the same time, there
are physical reasons to eliminate this gap.

The physical background of these derivational problems is the Bose-Einstein
condensate, also called the fifth state of matter, first experimentally
discovered in 1995 \cite{AEMWC,DMADDKK} after the prediction by {Einstein}.
In this context, the initial datum $\psi _{N}(0)$ of (\ref{eqn:N-body})
represents a trapped $N$-particle gas cooled very close to absolute zero
during the preparation phase and the dynamics $\psi _{N}(t)$ is the
evolution of the system during the observation phase after the confinement
is switched. That is, $\psi _{N}(0)$ is (or is very near) the ground state
of a $N$-body Schr\"{o}dinger operator with an external trapping potential
and hence its smoothness fully depends on the variable coefficients inside
the $N$-body Schr\"{o}dinger operator, which is mainly the trapping
potential in this case. In the original qualitative experiments \cite%
{AEMWC,DMADDKK}, the trap was generated by a strong magnetic field which is
smooth by definition. However, since around 1997, the experiments -- see 
\cite{SOT,SACIMSK} for examples -- have been instead favoring a pulse-type
laser trapping, as it produces less background noise for quantitative
measurement and gives more control of the parameters of the system. However,
due to the discrete / pulse nature and the complicated deployment of the
technology, such an optical confinement is not very smooth and can only be
approximated as harmonic $\omega _{0}^{2}\left\vert x\right\vert ^{2}$ when
far off.\footnote{%
See \cite{N1} for some locally half-circle shaped or paralleled-tube shaped
examples.} That is, away from the usual difficulties in measuring a high
Sobolev norm of a microscopic quantum mechanical system, the initial datum $%
\psi _{N}(0)$ of (\ref{eqn:N-body}) may not be very smooth at all due to the
setup of the system. On the other hand, it is always safe to assume the $%
H^{1}$ condition as every particle in the system must have finite kinetic
and potential energy which are also primary characteristics of the system.
It is, therefore, of substantial physical interest to close the
aforementioned regularity gap.

In this paper, we address the issue of the regularity gap using the
hierarchy method in the Klainerman-Machedon theory format and refining an
idea from the Fock space method. Let $S^{(\alpha
,k)}=\dprod\limits_{j=1}^{k}\left\langle \nabla _{x_{j}}\right\rangle
^{\alpha }\left\langle \nabla _{x_{j}^{\prime }}\right\rangle ^{\alpha }$ as
usual, we define our master norm $\left\Vert \cdot \right\Vert
_{H_{Z}^{\alpha }}$ for a hierarchy of marginal densities $\Gamma =\left\{
\gamma ^{(k)}\right\} _{k=1}^{\infty }$, following \cite{TCNP1}-\cite{TCKT},
by%
\begin{equation}
\left\Vert \Gamma \right\Vert _{H_{Z}^{\alpha }}=\sum_{k=1}^{\infty
}Z^{-k}\left\Vert S^{(\alpha ,k)}\gamma ^{(k)}\right\Vert _{L^{2}}.
\label{norm:master capital}
\end{equation}%
{We note that this norm is guaranteed to converge for $Z>C$ provided that
for all $k\geq 1$, $\Vert S^{(\alpha ,k)}\gamma ^{(k)}\Vert _{L^{2}}\leq
C^{k}$} and we will only use the norm when the condition $\Vert S^{(\alpha
,k)}\gamma ^{(k)}\Vert _{L^{2}}\leq C^{k}$ is known to hold, and thus we
call it a \textquotedblleft norm\textquotedblright . We will assume the
following usual conditions for our main theorem under norm (\ref{norm:master
capital}):

\begin{itemize}
\item[(a)] $\psi _{N}(0)$ is normalized, that is $\left\Vert \psi
_{N}(0)\right\Vert _{L^{2}(\mathbb{R}^{3N})}=1$ or $\limfunc{Tr}\gamma
_{N}^{(k)}(0)=1.$
\end{itemize}

\begin{enumerate}
\item[(b)] The uniform energy bounds hold:\footnote{%
One can use either (\ref{H1 condition:Energy}) or (\ref{H1 condition:H1}) as
(b).} 
\begin{equation}
\left\langle \psi _{N}(0),\left( H_{N}/N\right) ^{k}\psi
_{N}(0)\right\rangle \leqslant E_{0}^{k}  \label{H1 condition:Energy}
\end{equation}%
which, as {shown by \cite{E-S-Y2}, implies that for all $k\geq 1$ and all $t$%
, $\Vert S^{(1,k)}\gamma _{N}^{(k)}\Vert _{L^{2}}\leq 2^{k}E_{0}^{k}$, which
further implies} 
\begin{equation}
\sup_{t\in \left[ 0,T\right] }\left\Vert \Gamma _{N}(t)\right\Vert
_{H_{Z}^{1}}<+\infty .  \label{H1 condition:H1}
\end{equation}%
{for any $Z>2E_{0}$.}

\item[(c)] For some $Z_{0}>2E_{0}$, and for some $\phi _{0}\in H^{1}(\mathbb{%
R}^{3})$, the initial condition is asymptotically factorized: 
\begin{equation*}
\lim_{N\rightarrow \infty }\left\Vert \Gamma _{N}(0)-\left\{ \left\vert \phi
_{0}\right\rangle \left\langle \phi _{0}\right\vert ^{\otimes k}\right\}
\right\Vert _{H_{Z_{0}}^{1}}=0.
\end{equation*}
\end{enumerate}

Our main theorem is the following.

\begin{theorem}[Main Theorem]
\label{thm:convergence rate main theorem}Assume the marginal densities $%
\Gamma _{N}=\{\gamma _{N}^{(k)}\}$ associated with $\psi _{N}$, the solution
to the $N$-body dynamics (\ref{eqn:N-body}) with a smooth even pair
interaction $V\geqslant 0$, satisfy (a)-(c). Then for $T=c_{1}E_{0}^{-2}$
and $Z=c_{2}E_{0}$ for some specific multiples $c_{1}$, $c_{2}$, we have the
estimate 
\begin{equation}
\begin{aligned} \hspace{0.3in}&\hspace{-0.3in} \sup_{t\in \left[ 0,T\right]
}\left\Vert \Gamma _{N}(t)-\left\{ \left\vert \phi (t)\right\rangle
\left\langle \phi (t)\right\vert ^{\otimes k}\right\} \right\Vert
_{H_{Z}^{1}} \\ &\lesssim \|\Gamma_N(0)-\left\{ \left\vert \phi
_{0}\right\rangle \left\langle \phi _{0}\right\vert ^{\otimes k}\right\}
\|_{H^1_{Z/2}} + \max( N^{\frac52\beta-1}, N^{-\beta}(\ln N)^7) \end{aligned}
\label{estimate:convergence rate main estimate for H1 with exp in time}
\end{equation}%
where $\phi (t)$ solves (\ref{NLS:opening cubic}) with $\phi (0)=\phi _{0}$
and $b_{0}=\int V$.
\end{theorem}

The proof of Theorem \ref{thm:convergence rate main theorem} certainly
allows general datum as usual. In the context of the quantum de Finetti
theorem in \cite{CHPS} from \cite{LMR}, Theorem \ref{thm:convergence rate
main theorem} reads as follows.

\begin{corollary}[General Datum]
\label{cor:convergence rate main theorem}Assume the marginal densities $%
\Gamma _{N}=\{\gamma _{N}^{(k)}\}$ associated with $\psi _{N}$, the solution
to the $N$-body dynamics (\ref{eqn:N-body}) with $V\geqslant 0,$ satisfy
(a), (b) and

(c') For some $E_{0}>0$, and for some probability measure $d\mu _{0}$
supported on $\mathbb{S(}L^{2}(\mathbb{R}^{3}))$, we have 
\begin{equation*}
\lim_{N\rightarrow \infty }\left\Vert \Gamma _{N}(0)-\Gamma _{\infty
}(0)\right\Vert _{H_{E_{0}}^{1}}=0.
\end{equation*}%
where 
\begin{equation*}
\Gamma _{\infty }(0)=\left\{ \int_{\mathbb{S(}L^{2}(\mathbb{R}%
^{3}))}\left\vert \phi \right\rangle \left\langle \phi \right\vert ^{\otimes
k}d\mu _{0}(\phi )\right\}
\end{equation*}%
Then for $T=c_{1}E_{0}^{-2}$ and $Z=c_{2}E_{0}$ for some specific multiples $%
c_{1}$, $c_{2}$, we have the estimate 
\begin{equation*}
\sup_{t\in \left[ 0,T\right] }\left\Vert \Gamma _{N}(t)-\Gamma _{\infty
}(t)\right\Vert _{H_{Z}^{1}}\lesssim \Vert \Gamma _{N}(0)-\Gamma _{\infty
}(0)\Vert _{H_{Z/2}^{1}}+\max (N^{\frac{5}{2}\beta -1},N^{-\beta }(\ln
N)^{7})
\end{equation*}%
where 
\begin{equation*}
\Gamma _{\infty }(t)=\left\{ \int_{\mathbb{S(}L^{2}(\mathbb{R}%
^{3}))}\left\vert S_{t}\phi \right\rangle \left\langle S_{t}\phi \right\vert
^{\otimes k}d\mu _{0}(\phi )\right\}
\end{equation*}%
and $S_{t}:H^{1}(\mathbb{R}^{3})\rightarrow H^{1}(\mathbb{R}^{3})$, $\forall
t\in \mathbb{R}$,\footnote{%
Condition (c) implies $\Gamma _{\infty }(0)\in H_{E_{0}}^{1}$ which implies $%
d\mu _{0}$ is supported in the subset of $\mathbb{S(}L^{2}(\mathbb{R}^{3}))$
in which $\left\Vert \phi \right\Vert _{H^{1}(\mathbb{R}^{3})}\leqslant
E_{0} $. Hence $S_{t}\phi $ is well-defined inside the $d\mu _{0}$ integral.}
is the solution map of (\ref{NLS:opening cubic}).
\end{corollary}

For $\beta \in \left( 0,\frac{2}{5}\right) $, Theorem \ref{thm:convergence
rate main theorem} and Corollary \ref{cor:convergence rate main theorem}
give a convergence rate estimate.\footnote{%
The method does yield the optimal $N^{-1}$ rate when $\beta =0$ but one
needs to change (\ref{NLS:opening cubic}).} Also, to be precise, the
concluded rate is in the $H^{1}$ norm which is stronger than the usual trace
norm convergence. Of course, both are physically meaningful with one being
the kinetic energy and one being the probability when restricted to $\gamma
^{(1)}$. We remark that the rate $\frac{\left( \ln N\right) ^{7}}{N^{\beta }}%
,$ when $\beta \in \left( 0,\frac{2}{7}\right] ,$ is an almost optimal in $N$
rate\footnote{%
We emphasize the \textquotedblleft in $N$\textquotedblright\ aspect of the
optimality here because the best \textquotedblleft in $t$\textquotedblright\
growth rate is unknown. But a rate better than exponential growth has been
proven to be possible in related scenarios with the second-order correction
-- see \cite{Chen2ndOrder,JC1,GM1,GM2,GMM1,GMM2,Kuz,KP} for examples.} as
the optimal in $N$ rate is $\frac{1}{N^{\beta }}$ if we require both sides
of the estimates to be in $H^{1}$. With a more delicate argument, the power
of $\ln N$ can be reduced. We leave it at $7$ for simplicity as $\left( \ln
N\right) ^{7}$ is still better than $N^{\varepsilon }$. If we assume extra
regularity in the limiting initial datum $\phi _{0}$, then our rate improves
and in fact reveals more details of the story.

\begin{corollary}[Improved Rate with $H^{q}$, $q>1$ Datum]
\label{Cor:Optimal N-body rate}In addition to the assumptions (a)-(c),
assume that the limiting intial one-particle state $\phi _{0}\in H^{q}(%
\mathbb{R}^{3})$ for some $q>1$, then estimate (\ref{estimate:convergence
rate main estimate for >H1 with exp in time}) in Theorem \ref%
{thm:convergence rate main theorem} can be improved to 
\begin{eqnarray}
&&\sup_{t\in \left[ 0,T\right] }\left\Vert \Gamma _{N}(t)-\left\{ \left\vert
\phi (t)\right\rangle \left\langle \phi (t)\right\vert ^{\otimes k}\right\}
\right\Vert _{H_{Z}^{1}}
\label{estimate:convergence rate main estimate for >H1 with exp in time} \\
&\lesssim &_{q,\left\Vert \phi _{0}\right\Vert _{H^{q}}}\Vert \Gamma
_{N}(0)-\left\{ \left\vert \phi _{0}\right\rangle \left\langle \phi
_{0}\right\vert ^{\otimes k}\right\} \Vert _{H_{Z/2}^{1}}+\max (N^{\frac{5}{2%
}\beta -1},N^{-\min (q,2)\beta }(\ln N)^{7})  \notag
\end{eqnarray}
\end{corollary}

Corollary \ref{Cor:Optimal N-body rate} proves that we get an optimal $H^{1}$
rate if we have $H^{1+\varepsilon }$ regularity as $\frac{\left( \ln
N\right) ^{7}}{N^{(1+\varepsilon )\beta }}$ is better than $\frac{1}{%
N^{\beta }}$. On the other hand, apparently, the optimal $H^{1}$ rate
improves if we have $H^{q},$ $q>1$.

We remark that it is not too difficult to use some extra Littlewood-Paley
argument to improve the $N^{\frac{5\beta }{2}-1}$ inside estimates (\ref%
{estimate:convergence rate main estimate for H1 with exp in time}) and (\ref%
{estimate:convergence rate main estimate for >H1 with exp in time}) to $N^{%
\frac{3\beta }{2}-1}$ which concludes convergence for $\beta \in \left( 0,%
\frac{2}{3}\right) $ and yields the optimal in $N$ rate for $\beta \in
\left( 0,\frac{2}{5}\right] $. We choose not to do so in this paper as we
would like to keep this paper short. The $\frac{\left( \ln N\right) ^{7}}{%
N^{\beta }}$ and $\frac{\left( \ln N\right) ^{7}}{N^{q\beta }}$ in estimates
(\ref{estimate:convergence rate main estimate for H1 with exp in time}) and (%
\ref{estimate:convergence rate main estimate for >H1 with exp in time}) come
from the following new NLS \textquotedblleft
bi-scattering\textquotedblright\ result. To state it, let $\phi _{N}$ solve
the Hartree type NLS (H-NLS) equation 
\begin{equation}
\begin{aligned} &i\partial _{t}\phi _{N} =-\Delta \phi _{N}+(V_{N}\ast |\phi
_{N}|^{2})\phi _{N}\text{ in }\mathbb{R}^{3+1} \\ &\phi _{N}(0,x) =\phi
_{0}(x) \end{aligned}  \label{NLS:Hartree}
\end{equation}%
where $V_{N}(x)=N^{3\beta }V(N^{\beta }x)$ as in (\ref{hamitonian:N-body})
with $V\geqslant 0$.

\begin{theorem}[bi-scattering]
\label{Thm:biScattering} Let $\phi$ solve \eqref{NLS:opening cubic} and $%
\phi_N$ solve \eqref{NLS:Hartree}. We then have the following.

(i) For $H^1$ data, both $\phi(t)$ and $\phi_N(t)$ satisfy the
global-in-time bounds 
\begin{equation*}
\Vert \langle \nabla \rangle \phi \Vert _{L_{t}^{2}(\mathbb{R}%
)L_{x}^{6}}\lesssim _{\left\Vert \phi _{0}\right\Vert _{H^{1}}}1\,,\qquad
\Vert \langle \nabla \rangle \phi _{N}\Vert _{L_{t}^{2}(\mathbb{R}%
)L_{x}^{6}}\lesssim _{\left\Vert \phi _{0}\right\Vert _{H^{1}}}1,
\end{equation*}
and scatter in $H^1$ -- that is, there exist forward-in-time states $\phi_+$
and $\phi_{N,+}$ such that 
\begin{equation*}
\lim_{t \to \infty} \| \phi(t)- e^{it\Delta} \phi_+ \|_{H_x^1} = 0 \,,
\qquad \lim_{t \to \infty} \| \phi_N(t)- e^{it\Delta} \phi_{N,+} \|_{H_x^1}
= 0
\end{equation*}

(ii) We have the global-in-time $H^{1}$ comparison estimate 
\begin{equation}
\left\Vert \phi -\phi _{N}\right\Vert _{L_{t}^{\infty }(\mathbb{R}%
)H_{x}^{1}}\lesssim _{q,\left\Vert \phi _{0}\right\Vert _{H^{q}}}\frac{%
\left( \ln N\right) ^{7}}{N^{q\beta }}  \label{nls difference:almost optimal}
\end{equation}%
provided that $\phi _{0}\in H^{q}$, $q\in \left[ 1,2\right] $, and the
optimal rate is $N^{-q\beta }$.
\end{theorem}

We say Theorem \ref{Thm:biScattering} is a \textquotedblleft
bi-scattering\textquotedblright\ result not because, as stated in (i), both (%
\ref{NLS:opening cubic}) and (\ref{NLS:Hartree}) scatter, which is in fact
known, but because the conclusion in (ii) that the interaction potential $%
V_{N}\rightarrow \delta $ as $N\rightarrow \infty $ for two corresponding
Hamiltonian evolution is called a $2$-body scattering process in the context
of quantum many-body dynamics. Moreover, estimate (\ref{nls
difference:almost optimal}) holds globally and thus carries some $%
t\rightarrow \infty $ information. That is, one scattering is the usual $%
t\rightarrow \infty $ scattering while another scattering is the $%
N\rightarrow \infty $ scattering, and they happen simultaneously as in (\ref%
{nls difference:almost optimal}).

On the one hand, we prove Theorem \ref{Thm:biScattering} which is usually an
ingredient in a Fock space approach paper. On the other hand, our proof does
not have a compactness or uniqueness argument as in the standard hierarchy
approach. One could, in fact, view the main proof of this paper as
integrating the idea from Fock space approach that, using (\ref{NLS:Hartree}%
) as an intermediate dynamic, into the hierarchy method in the
Klainerman-Machedon theory format. The fact that we close the regularity gap
and prove the (almost) optimal rates with such a simple combination is
exactly the main novelty of this paper. Let us now give a brief review of
the two approaches.

Limit (\ref{limit:mean field limit}) was first established in the work of Erd%
\"{o}s, Schlein, and Yau \cite{E-S-Y2,E-S-Y5,E-S-Y3} for the $\mathbb{R}^{3}$
defocusing cubic case around 2005.\footnote{%
See also \cite{AGT} for the 1D defocusing cubic case around the same time.}
They first proved (\ref{H1 condition:Energy}) implies (\ref{H1 condition:H1}%
) as a preparation. They then proved that $\left\{ \Gamma _{N}(t)\right\} $
is a compact sequence with respect to a suitable weak* topology on trace
class operators using the fact $\Gamma _{N}(t)$ satisfies the
Bogoliubov--Born--Green--Kirkwood--Yvon (BBGKY) hierarchy 
\begin{eqnarray}
i\partial _{t}\gamma _{N}^{(k)} &=&\sum_{j=1}^{k}\left[ -\Delta
_{x_{j}},\gamma _{N}^{(k)}\right] +\frac{1}{N}\sum_{1\leqslant i<j\leqslant
k}\left[ V_{N}(x_{i}-x_{j}),\gamma _{N}^{(k)}\right]  \label{hierarchy:BBGKY}
\\
&&+\frac{N-k}{N}\sum_{j=1}^{k}\limfunc{Tr}\nolimits_{k+1}\left[
V_{N}(x_{j}-x_{k+1}),\gamma _{N}^{(k+1)}\right] ,  \notag
\end{eqnarray}%
and that every limit point $\Gamma _{\infty }(t)=\left\{ \gamma _{\infty
}^{(k)}\right\} $ of $\left\{ \Gamma _{N}(t)\right\} $ satisfies the
Gross-Pitaevskii (GP) hierarchy%
\begin{equation}
i\partial _{t}\gamma _{\infty }^{(k)}=\sum_{j=1}^{k}\left[ -\Delta
_{x_{k}},\gamma _{\infty }^{(k)}\right] +b_{0}\sum_{j=1}^{k}\limfunc{Tr}%
\nolimits_{k+1}\left[ \delta (x_{j}-x_{k+1}),\gamma _{\infty }^{(k+1)}\right]
.  \label{hierarchy:GP}
\end{equation}%
Finally, they proved delicatedly that there is a unique solution to the $%
\mathbb{R}^{3}$ cubic GP hierarchy in a $H^{1}$-type space (unconditional
uniqueness) in \cite{E-S-Y2} with a sophisticated Feynman graph analysis and
many highly technical singular integral techniques. Because the desired
limit $\{\left\vert \phi (t)\right\rangle \left\langle \phi (t)\right\vert
^{\otimes k}\}$ solves hierarchy (\ref{hierarchy:GP}), limit (\ref%
{limit:mean field limit}) is then proved without any rate estimate. This
first series of ground breaking papers have motivated a large amount of
work. Moreover, in \cite{E-S-Y5,E-S-Y3}, the weak* convergence was upgraded
to strong via an elementary functional analysis theorem. This
\textquotedblleft small\textquotedblright\ weak* to strong upgrade firmly
hinted that a convergence rate result is possible.

In 2007, Klainerman and Machedon \cite{KM}, inspired by \cite%
{E-S-Y2,KMNullForm}, proved the uniqueness of solutions regarding (\ref%
{hierarchy:GP}) in a Strichartz-type space (conditional uniqueness). They
proved a collapsing type estimate, to estimate the inhomogeneous term in (%
\ref{hierarchy:GP}), and provided a different combinatorial argument, the
now so-called Klainerman-Machedon (KM) board game, to combine the
inhomogeneous terms effectively reducing their numbers. At that time, it was
unknown how to prove that the limits coming from (\ref{hierarchy:BBGKY}) are
in the Strichartz type spaces even though the target limit $\{\left\vert
\phi (t)\right\rangle \left\langle \phi (t)\right\vert ^{\otimes k}\}$
generated by (\ref{NLS:opening cubic}) naturally lie in both the $H^{1}$%
-type space and the Strichartz type space. Nonetheless, \cite{KM} has made
the analysis of (\ref{hierarchy:GP}) approachable to PDE analysts and the KM
board game has been used in every work involving hierarchy (\ref%
{hierarchy:GP}).

When Kirkpatrick, Schlein, and Staffilani \cite{KSS} found that the KM
Strichartz-type bound can be obtained via a simple trace theorem for the
defocusing case in $\mathbb{R}^{2}$ and $\mathbb{T}^{2}$ in 2008, many works 
\cite{TCNP2,ChenAnisotropic,C-HFocusing,C-HFocusingII,GSS,HS1,SH1,S,Xie}
then followed such a scheme for the uniqueness of GP hierarchies. However,
how to check the KM bound in the 3D cubic case remained fully open at that
time.

T. Chen and Pavlovi\'{c} studied the 1D and 2D defocusing quintic case and
laid the foundation for the 3D quintic defocusing energy-critical case in
their late 2008 work \cite{TCNP2}, in which they proved that the 2D quintic
case, a case usually considered equivalent to the 3D cubic case, does
satisfy the KM bound though proving it for the 3D cubic case was still open.

In \cite{TCNP1,TCNP3,TCNP4}, T. Chen and Pavlovi\'{c} generalized the
problem and launched the well-posedness theory of (\ref{hierarchy:GP}) with
general initial datum as an independent subject away from (\ref{eqn:N-body}%
). (See also \cite{TCNPNT,MNPS,MNPRS1,MNPRS2,S,SS}.) Then in 2011, T. Chen
and Pavlovic proved the 3D cubic KM Strichartz type bound for the defocusing 
$\beta <1/4$ case in \cite{TCNP5}. (See also \cite{TCKT}.) The result was
quickly improved to $\beta \leqslant 2/7$ by X.C. in \cite{Chen3DDerivation}
and to the almost optimal case, $\beta <1,$ by X.C. and J.H. in \cite%
{C-H2/3,C-H<1}, by lifting the $X_{1,b}$ space techniques from NLS theory
into the field. Away from being the 1st work to prove the KM bound, the work 
\cite{TCNP5}, in fact, hinted\footnote{%
Private communication in 2011.} two unforeseen research directions of the
hierarchy method today.

One direction is to prove new NLS results via the more general but at the
same time more complicated hierarchy (\ref{hierarchy:GP}). The hierarchy
uniqueness theorems started to match the corresponding NLS results in \cite%
{S1,HTX,HTX1,C-HFocusingII,C-PUniqueness} following the 2013 introduction of
the quantum de Finetti theorem from \cite{LMR} to the field by T. Chen,
Hainzl, Pavlovi\'{c}, and Seiringer \cite{CHPS}. Then, recently, the
previously open $\mathbb{T}^{d}$ NLS unconditional uniqueness problems
either saw substantial progress or were solved via the analysis of the
supposedly more complicated GP hierarchy. In \cite{HS2}, Herr and Sohinger
generalized the Sobolev multilinear estimates in \cite{CHPS} and obtained
new unconditional uniqueness results regarding $\mathbb{T}^{d}$ GP hierarchy
and hence $\mathbb{T}^{d}$ NLS. In \cite{CJInvent}, by discovering the
hierarchical uniform frequency localization (HUFL) property, X.C. and J.H.
established, for the $\mathbb{T}^{3}$ quintic energy-critical GP hierarchy,
a $H^{1}$-type uniqueness theorem which was neither conditional nor
unconditional but implies the $H^{1}$ unconditional uniqueness for the $%
\mathbb{T}^{3}$ quintic energy-critical NLS. More recently, in \cite{CJT4},
X.C. and J.H. worked out an extended KM board game from scratch to enable,
finally, the application of dispersive norms like $U$-$V$ and $X_{s,b}$ in
the field and proved the $H^{1}$ unconditional uniqueness for the $\mathbb{T}%
^{4}$ cubic energy-critical NLS, an unanticipated \textquotedblleft
special\textquotedblright\ case in the $\mathbb{R}^{3}/\mathbb{R}^{4}/%
\mathbb{T}^{3}/\mathbb{T}^{4}$ energy-critical sequence, with the hierarchy
approach. The proof in \cite{CJT4} went so smoothly that, X.Chen, Shen, and
Zhang completely and unifiedly solved the unconditional uniqueness for $%
\mathbb{R}^{d}$ and $\mathbb{T}^{d}$ cubic and quintic energy-supercritical
NLS in \cite{SH2}.

The other direction hinted in \cite{TCNP5} is that it is possible to use the
KM theory to construct a hierarchy method without the compactness argument,
that is, it is possible to establish convergence rate estimates using the
hierarchy approach with $H^{1}$ regularity. In other words, \cite%
{TCNP5,Chen3DDerivation,C-H2/3,C-H<1}, can be considered as premodels of
this paper. However, completing the proof directly using solely the
hierarchy approach in KM format will need to pass, on the road, some extra
technical difficulties, like extra error terms, which charge a $N^{-\frac{%
\beta }{2}}$ price. On the other hand, in the Fock space approach, equation (%
\ref{NLS:Hartree}) naturally pops out 1st in the 2nd quantization argument
and one always needs to compare between equations (\ref{NLS:opening cubic})
and (\ref{NLS:Hartree}) to close limit (\ref{limit:mean field limit}), see,
for example \cite{BOB,BCS,BS,JC1,GM1,GM2,Kuz,NS}.\footnote{%
This is certainly only a fraction of all possible references as the Fock
space approach is also such a vast and sophisticated subject now. Please
also see the references within them and the newer ones online.} We
assimilate this idea into our proof.

\subsection{Outline of the Proof}

As the proof of Corollary \ref{cor:convergence rate main theorem} only
requires adding the $d\mu _{0}$ integrals in suitable places, we prove only
Theorem \ref{thm:convergence rate main theorem} and Corollary \ref%
{Cor:Optimal N-body rate} in detail. In fact, we stated Corollary \ref%
{cor:convergence rate main theorem} not to show the generality, but to
clarify a logical question that if the proof for Theorem \ref%
{thm:convergence rate main theorem} relies on any uniqueness theorems
regarding (\ref{hierarchy:GP}). For Theorem \ref{thm:convergence rate main
theorem} and Corollary \ref{Cor:Optimal N-body rate}, we indeed did not use
any uniqueness results regarding (\ref{hierarchy:GP}) as the desired limit
could be guessed in multiple ways. (The 2nd quantization argument is
certainly an option.) But for the general datum case, Corollary \ref%
{cor:convergence rate main theorem}, we are not aware of any method to guess
the desired limit away from using uniqueness results regarding (\ref%
{hierarchy:GP}) in \cite{CHPS}.

As mentioned before, the main proof here can be understood as utilizing the
hierarchy approach in KM format but put (\ref{NLS:Hartree}) as an
intermediate dynamic. That is, we prove estimates (\ref{estimate:convergence
rate main estimate for H1 with exp in time}) and (\ref{estimate:convergence
rate main estimate for >H1 with exp in time}) by summing two estimates. The
1st one is%
\begin{eqnarray}
&&\sup_{t\in \left[ 0,T_{0}\right] }\left\Vert \Gamma _{N}(t)-\left\{
\left\vert \phi _{N}(t)\right\rangle \left\langle \phi _{N}(t)\right\vert
^{\otimes k}\right\} \right\Vert _{H_{Z_{1}}^{1}}
\label{estimate:key-1 for almost optimal} \\
&\lesssim &_{E_{0}}\left\Vert \Gamma _{N}(0)-\left\{ \left\vert \phi
_{0}\right\rangle \left\langle \phi _{0}\right\vert ^{\otimes k}\right\}
\right\Vert _{H_{Z_{1}/2}^{1}}+\frac{1}{N^{1-\frac{5\beta }{2}}},  \notag
\end{eqnarray}%
while the 2nd one is 
\begin{equation}
\sup_{t\in \mathbb{R}}\left\Vert \left\{ \left\vert \phi
_{N}(t)\right\rangle \left\langle \phi _{N}(t)\right\vert ^{\otimes
k}\right\} -\left\{ \left\vert \phi (t)\right\rangle \left\langle \phi
(t)\right\vert ^{\otimes k}\right\} \right\Vert _{H_{Z_{2}}^{1}}\lesssim
_{q,\left\Vert \phi _{0}\right\Vert _{H^{q}}}\frac{\left( \ln N\right) ^{7}}{%
N^{q\beta }}.  \label{estimate:key-2 for almost optimal}
\end{equation}%
Estimates (\ref{estimate:convergence rate main estimate for H1 with exp in
time}) and (\ref{estimate:convergence rate main estimate for >H1 with exp in
time}) then follow by selecting $Z=\max (Z_{1},Z_{2})$.

We prove estimate (\ref{estimate:key-1 for almost optimal}) in \S \ref%
{sec:bbgky vs Hnls} by directly taking the difference between (\ref%
{hierarchy:BBGKY}) and the \textquotedblleft H-NLS\textquotedblright\
hierarchy (\ref{hierarchy:Hartree}) generated by (\ref{NLS:Hartree}). We
iterate the difference hierarchy (\ref{hierarchy:Hartree difference}) by
coupling into the next level multiple times and group the terms into the
free part, driving part, and interaction part. We can then proceed to
estimate following the scheme in \cite{Chen3DDerivation,C-H2/3,C-H<1}. This
part is new but the method is not. We comment that another option would be
taking the direct difference between (\ref{hierarchy:BBGKY}) and (\ref%
{hierarchy:GP}). The problem is that such a route would have produced a
difference hierarchy with two options to couple to the next level, namely $%
\limfunc{Tr}\nolimits_{k+1}\left[ V_{N}(x_{j}-x_{k+1}),\gamma _{N}^{(k+1)}%
\right] $ and $\limfunc{Tr}\nolimits_{k+1}\left[ \delta
(x_{j}-x_{k+1}),\gamma _{\infty }^{(k+1)}\right] $, compared to (\ref%
{hierarchy:Hartree}) in which there is only one interaction term. While
iterating the hierarchy is basically the only way to obtain hierarchy
estimates since the beginning of the hierarchy approach, it is evident that
the \textquotedblleft more direct\textquotedblright\ route indeed has way
more error terms. Finally, one also needs to face the classical
\textquotedblleft trace vs power\textquotedblright\ technical dilemma
without the intermediate dynamic. This is the technical reason that we chose
to use (\ref{NLS:Hartree}) as an intermediate dynamic. But indeed, one can
still get a $N^{-\frac{\beta }{2}}$ rate with KM theory alone.

We then prove estimate (\ref{estimate:key-2 for almost optimal}) by proving
Theorem \ref{Thm:biScattering} in \S \ref{sec:Hnls vs NLS}. Though (\ref%
{NLS:opening cubic}) is $H^{\frac{1}{2}}$ critical instead of $H^{1}$
critical, the error estimate (\ref{nls difference:almost optimal}) yielding $%
N^{-q\beta }$ is critical in the sense that all spatial derivatives and
space-time H\"{o}lder norms are fully absorbed in the estimates. That is,
one has no choice but to use the $U$-$V$ spaces for the $N\rightarrow \infty 
$ scattering proof. It then forces the $t\rightarrow \infty $ argument to be
in the $U$-$V$ spaces as well. Theorem \ref{Thm:biScattering} is the 1st
bi-scattering theorem of its type and it is obvious by this paper that it
has direct applications. Another highlight of the proof is that it employs
the full strength of the $\mathbb{R}^{d}$ Schr\"{o}dinger bilinear estimate.
Finally, we provide the first proof of the optimality of the $N^{-\beta }$
rate, which has been mentioned multiple times with physical insight in the
literature, via the method of space-time resonance in \S \ref{sec:optimality}%
.

Putting \S \ref{sec:bbgky vs Hnls} and \S \ref{sec:Hnls vs NLS} together
concludes the proof of the main theorems, Theorem \ref{thm:convergence rate
main theorem} and Corollary \ref{Cor:Optimal N-body rate}. It is surprising
that under the simple scheme in this paper, without too much extra work, the
hierarchy approach yields convergence rate estimate which was obtainable, so
far, only via the Fock space approach. Moreover, it eliminates the $H^{1}$
vs $H^{4}$ regularity gap by requiring $H^{1}$ or $H^{1+\varepsilon }$
regularity. It is also astonishing that the almost optimal or optimal in $N$
convergence rate can be obtained with this easy method.\footnote{%
Obtaining the optimal $N^{-\beta }$ rate using the Fock space approach
assuming $H^{4}$ has been done in the much harder $\beta =1$ case. See \cite%
{BS}.} The discovery of this simple hierarchy approach is the main novelity
of this paper.

With the help of the extended KM board game in \cite{CJT4} which allows the
application of the $U$-$V$ space-time estimates, and a $X^{\frac{1}{2}+}$
frequency localized version of the KM estimates, we expect improving (\ref%
{estimate:key-1 for almost optimal}) up to $\beta =1$ once we put in the
correlation structures we had in \cite{C-H<1}. (Of course, (\ref{NLS:Hartree}%
) has to be changed as well.) We shall do so in the next (longer) paper. The
main argument of this paper can also be extended to any finite time but with
a $1/\ln N$ rate. (See \cite{CSWZ}.)

\section{Comparing the BBGKY hierarchy and the H-NLS\label{sec:bbgky vs Hnls}%
}

The main goal in this section is to prove (\ref{estimate:key-1 for almost
optimal}) which will result from Theorem \ref{Thm:BBGKY Estimates in
Masternorm}. We adopt the shorthands 
\begin{eqnarray*}
U^{(k)}(t) &=&e^{it\triangle _{\mathbf{x}_{k}}}e^{-it\triangle _{\mathbf{x}%
_{k}^{\prime }}}, \\
V_{N}^{(k)}\gamma _{N}^{(k)} &=&\frac{1}{N}\sum_{1\leqslant i<j\leqslant k}%
\left[ V_{N}(x_{i}-x_{j}),\gamma _{N}^{(k)}\right] , \\
B_{N}^{(k+1)}\gamma _{N}^{(k+1)} &=&\sum_{j=1}^{k}B_{N,j,k+1}\gamma
_{N}^{(k+1)}=\sum_{j=1}^{k}\limfunc{Tr}\nolimits_{k+1}[V_{N}(x_{j}-x_{k+1}),%
\gamma _{N}^{(k+1)}].
\end{eqnarray*}%
and assume $\int V=1$ for convenience. We start by rewriting the 3D cubic
BBGKY hierarchy (\ref{hierarchy:BBGKY}) in integral form 
\begin{eqnarray}
\gamma _{N}^{(k)}(t_{k}) &=&U^{(k)}(t_{k})\gamma
_{N,0}^{(k)}+\int_{0}^{t_{k}}U^{(k)}(t_{k}-t_{k+1})V_{N}^{(k)}\gamma
_{N}^{(k)}(t_{k+1})dt_{k+1}  \label{equation:short duhamel of BBGKY} \\
&&+\frac{N-k}{N}\int_{0}^{t_{k}}U^{(k)}(t_{k}-t_{k+1})B_{N}^{(k+1)}\gamma
_{N}^{(k+1)}(t_{k+1})dt_{k+1}.  \notag
\end{eqnarray}%
where we have omitted the $(-i)$ in front of the 2nd and the 3rd term in the
right side of (\ref{equation:short duhamel of BBGKY}) as usual as we are
going to put everything in absolute values. In addition to (\ref%
{equation:short duhamel of BBGKY}), for $k=1,...,N,...,$ we consider the
H-NLS hierarchy%
\begin{equation}
\gamma _{H}^{(k)}(t_{k})=U^{(k)}(t_{k})\gamma
_{0}^{(k)}+\int_{0}^{t_{k}}U^{(k)}(t_{k}-t_{k+1})B_{N}^{(k+1)}\gamma
_{H}^{(k+1)}(t_{k+1})dt_{k+1}\text{, }  \label{hierarchy:Hartree}
\end{equation}%
generated by $\left\{ \gamma _{H}^{(k)}\left( t_{k},\mathbf{x}_{k},\mathbf{x}%
_{k}^{\prime }\right) =\left\vert \phi _{N}\right\rangle \left\langle \phi
_{N}\right\vert ^{\otimes k}\right\} $, the tensor products of solutions to (%
\ref{NLS:Hartree}).

The main concern in this section is the difference $\omega
_{N,H}^{(k)}=\gamma _{N}^{(k)}-\gamma _{H}^{(k)}$ which solves the hierarchy%
\begin{eqnarray}
\omega _{N,H}^{(k)}(t_{k}) &=&U^{(k)}(t_{k})\omega
_{N,0}^{(k)}+\int_{0}^{t_{k}}U^{(k)}(t_{k}-t_{k+1})V_{N}^{(k)}\gamma
_{N}^{(k)}(t_{k+1})dt_{k+1}  \label{hierarchy:Hartree difference with error}
\\
&&-\frac{k}{N}\int_{0}^{t_{k}}U^{(k)}(t_{k}-t_{k+1})B_{N}^{(k+1)}\gamma
_{N}^{(k+1)}(t_{k+1})dt_{k+1}  \notag \\
&&+\int_{0}^{t_{k}}U^{(k)}(t_{k}-t_{k+1})B_{N}^{(k+1)}\omega
_{N,H}^{(k+1)}(t_{k+1})dt_{k+1}  \notag
\end{eqnarray}%
Of course, we are using the convention that $\gamma _{N}^{(k)}=0$ if $k>N$
in (\ref{hierarchy:Hartree difference with error}).

As the error term 
\begin{equation*}
\frac{k}{N}\int_{0}^{t_{k}}U^{(k)}(t_{k}-t_{k+1})B_{N}^{(k+1)}\gamma
_{N}^{(k+1)}(t_{k+1})dt_{k+1}
\end{equation*}%
in (\ref{hierarchy:Hartree difference with error}) can be handled by adding
an extra $\frac{\ln N}{N}$ to our estimates of the main terms, we can assume
it drops out\footnote{%
Intereseted readers can see \cite{CSWZ} for a detailed handling of this
error term.} and rewrite (\ref{hierarchy:Hartree difference with error}) as 
\begin{eqnarray}
\omega _{N,H}^{(k)}(t_{k}) &=&U^{(k)}(t_{k})\omega
_{N,0}^{(k)}+\int_{0}^{t_{k}}U^{(k)}(t_{k}-t_{k+1})V_{N}^{(k)}\gamma
_{N}^{(k)}(t_{k+1})dt_{k+1}  \label{hierarchy:Hartree difference} \\
&&+\int_{0}^{t_{k}}U^{(k)}(t_{k}-t_{k+1})B_{N}^{(k+1)}\omega
_{N,H}^{(k+1)}(t_{k+1})dt_{k+1}.  \notag
\end{eqnarray}%
Iterating hierarchy (\ref{hierarchy:Hartree difference}) $\ell _{c}$ times,
we have 
\begin{eqnarray}
&&\omega _{N,H}^{(k)}(t_{k})  \notag \\
&=&U^{(k)}(t_{k})\omega
_{N,0}^{(k)}+%
\int_{0}^{t_{k}}U^{(k)}(t_{k}-t_{k+1})B_{N}^{(k+1)}U^{(k+1)}(t_{k+1})\omega
_{N,0}^{(k+1)}dt_{k+1}  \notag \\
&&+\int_{0}^{t_{k}}U^{(k)}(t_{k}-t_{k+1})V_{N}^{(k)}\gamma
_{N}^{(k)}(t_{k+1})dt_{k+1}  \notag \\
&&+\int_{0}^{t_{k}}U^{(k)}(t_{k}-t_{k+1})B_{N}^{(k+1)}  \notag \\
&&\times \int_{0}^{t_{k+1}}U^{(k+1)}(t_{k+1}-t_{k+2})V_{N}^{(k+1)}\gamma
_{N}^{(k+1)}(t_{k+2})dt_{k+2}dt_{k+1}  \notag \\
&&+\int_{0}^{t_{k}}U^{(k)}(t_{k}-t_{k+1})B_{N}^{(k+1)}  \notag \\
&&\times \int_{0}^{t_{k+1}}U^{(k+1)}(t_{k+1}-t_{k+2})B_{N}^{(k+2)}\omega
_{N,H}^{(k+2)}(t_{k+2})dt_{k+2}dt_{k+1}  \notag \\
&=&...  \notag \\
&\equiv &\text{\textsc{FP}}^{(k,\ell _{c})}(t_{k})+\text{\textsc{DP}}%
^{(k,\ell _{c})}(t_{k})+\text{\textsc{IP}}^{(k,\ell _{c})}(t_{k}).
\label{E:decomp-01}
\end{eqnarray}%
where we have grouped the terms in $\omega _{N,H}^{(k)}(t_{k})$ into three
parts.

To write out the three parts of $\omega _{N,H}^{(k)}$, we define, the
notation that, for $j\geqslant 1$, 
\begin{eqnarray*}
&&J_{N}^{(k,j)}(\underline{t}_{(k,j)})(f^{(k+j)}) \\
&=&\left( U^{(k)}(t_{k}-t_{k+1})B_{N}^{(k+1)}\right) \cdots \left(
U^{(k+j-1)}(t_{k+j-1}-t_{k+j})B_{N}^{(k+j)}\right) f^{(k+j)},
\end{eqnarray*}%
and $J_{N}^{(k,0)}(\underline{t}_{k})(f^{(k)})=f^{(k)}(t_{k})$, where $%
\underline{t}_{(k,j)}$ means $\left( t_{k+1},\ldots ,t_{k+j}\right) $ for $%
j\geqslant 1$ and $t_{k}$ for $j=0$. In this notation, the \emph{free part}
of $\omega _{N,H}^{(k)}$ at coupling level $\ell _{c}$ is given by 
\begin{eqnarray*}
\QTR{sc}{FP}^{(k,\ell _{c})} &=&U^{(k)}(t_{k})\omega _{N,0}^{(k)}+ \\
&&\sum_{j=1}^{\ell _{c}}\int_{0}^{t_{k}}\cdots
\int_{0}^{t_{k+j-1}}U^{(k)}(t_{k}-t_{k+1})B_{N}^{(k+1)}\cdots \\
&&\times U^{(k+j-1)}(t_{k+j-1}-t_{k+j})B_{N}^{(k+j)}\left(
U^{(k+j)}(t_{k+j})\omega _{N,0}^{(k+j)}\right) d\underline{t}_{(k,j)} \\
&=&\sum_{j=0}^{\ell _{c}}\int_{0}^{t_{k}}\cdots
\int_{0}^{t_{k+j-1}}J_{N}^{(k,j)}(\underline{t}%
_{(k,j)})(f_{FP}^{(k,j)})(t_{k+j})d\underline{t}_{(k,j)}
\end{eqnarray*}%
where in the $j=0$ case, it is meant that there are no time integrals and $%
J^{(k,0)}$ is the identity operator, and 
\begin{equation*}
f_{FP}^{(k,j)}(t_{k+j})=U^{(k+j)}(t_{k+j})\omega _{N,0}^{(k+j)}\text{;}
\end{equation*}%
the \emph{driving part}, which is a forcing term for $\omega _{N,H}^{(k)}$
but a potential term for $\gamma _{N}^{(k)}$, is given by 
\begin{eqnarray*}
&&\QTR{sc}{DP}^{(k,\ell _{c})} \\
&=&\int_{0}^{t_{k}}U^{(k)}(t_{k}-t_{k+1})V_{N}^{(k)}\gamma
_{N}^{(k)}(t_{k+1})dt_{k+1}+ \\
&&\sum_{j=1}^{\ell _{c}}\int_{0}^{t_{k}}\cdots
\int_{0}^{t_{k+j-1}}U^{(k)}(t_{k}-t_{k+1})B_{N}^{(k+1)}\cdots
U^{(k+j-1)}(t_{k+j-1}-t_{k+j}) \\
&&\times
B_{N}^{(k+j)}(\int_{0}^{t_{k+j}}U^{(k+j)}(t_{k+j}-t_{k+j+1})V_{N}^{(k+j)}%
\gamma _{N}^{(k+j)}(t_{k+j+1})dt_{k+j+1})d\underline{t}_{(k,j)} \\
&=&\sum_{j=0}^{\ell _{c}}\int_{0}^{t_{k}}\cdots
\int_{0}^{t_{k+j-1}}J_{N}^{(k,j)}(\underline{t}%
_{(k,j)})(f_{DP}^{(k,j)})(t_{k+j})d\underline{t}_{(k,j)},
\end{eqnarray*}%
where in the $j=0$ case, it is meant that there are no time integrals and $%
J^{(k,0)}$ is the identity operator, and 
\begin{equation}
f_{DP}^{(k,j)}(t_{k+j})=%
\int_{0}^{t_{k+j}}U^{(k+j)}(t_{k+j}-t_{k+j+1})V_{N}^{(k+j)}\gamma
_{N}^{(k+j)}(t_{k+j+1})dt_{k+j+1};  \label{eqn:def of f inside PP}
\end{equation}%
and the \emph{interaction part} is given by 
\begin{eqnarray*}
\QTR{sc}{IP}^{(k,\ell _{c})} &=&\int_{0}^{t_{k}}\cdots \int_{0}^{t_{k+\ell
_{c}}}dt_{k+1}\cdots dt_{k+\ell _{c}+1}\text{ }%
U^{(k)}(t_{k}-t_{k+1})B_{N}^{(k+1)}\cdots \\
&&\qquad \cdots U^{(k+\ell _{c})}(t_{k+\ell _{c}}-t_{k+\ell
_{c}+1})B_{N}^{(k+\ell _{c}+1)}\left( \omega _{N,H}^{(k+\ell
_{c}+1)}(t_{k+\ell _{c}+1})\right) \\
&=&\int_{0}^{t_{k}}\cdots \int_{0}^{t_{k+\ell _{c}}}J_{N}^{(k,\ell _{c}+1)}(%
\underline{t}_{(k,\ell _{c}+1)})\left( \omega _{N,H}^{(k+\ell
_{c}+1)}(t_{k+\ell _{c}+1})\right) d\underline{t}_{(k,\ell _{c}+1)}.
\end{eqnarray*}%
Notice that, on the one hand, the $\QTR{sc}{FP}^{(k,\ell _{c})}$ and $%
\QTR{sc}{DP}^{(k,\ell _{c})}$ are sums while $\QTR{sc}{IP}^{(k,\ell _{c})}$
is a single term; on the other hand, $\QTR{sc}{DP}^{(k,\ell _{c})}$ depends
solely on $\gamma _{N}^{(k)}$ and is independent of $\gamma _{H}^{(k)},$
while $\QTR{sc}{FP}^{(k,\ell _{c})}$ and $\QTR{sc}{IP}^{(k,\ell _{c})}$
depends on $\gamma _{H}^{(k)}$. We have the following estimates.

\begin{proposition}
\label{Thm:Estimates for FP,PP,IP in H1} We have the following estimates.
For the free part, 
\begin{equation}
\left\Vert S^{(1,k)}\text{\textsc{FP}}^{(k,\ell _{c})}\right\Vert
_{L_{t_{k}}^{\infty }\left[ 0,T\right] L_{x,x^{\prime }}^{2}}\leqslant
\sum_{j=0}^{\ell _{c}}2^{k}(4CT^{1/2})^{j}\left\Vert S^{(1,k+j)}\omega
_{N,H}^{(k+j)}(0)\right\Vert _{L_{x,x^{\prime }}^{2}}  \label{E:decomp-02}
\end{equation}%
Provided $T\lesssim E_{0}^{-2}$, the driving part satisfies 
\begin{equation}
\left\Vert S^{(1,k)}\text{\textsc{DP}}^{(k,\ell _{c})}\right\Vert
_{L_{t_{k}}^{\infty }\left[ 0,T\right] L_{x,x^{\prime }}^{2}}
\label{E:decomp-03} \\
\leqslant CT^{1/2}N^{\frac{5}{2}\beta -1}(2E_{0})^{k}k^{2}
\end{equation}%
For the interaction part, we have 
\begin{eqnarray}
&&\left\Vert S^{(1,k)}\text{\textsc{IP}}^{(k,\ell _{c})}\right\Vert
_{L_{t_{k}}^{\infty }\left[ 0,T\right] L_{x,x^{\prime }}^{2}}
\label{E:decomp-05} \\
&\leqslant &2^{k}T^{1/2}(4CT^{1/2})^{\ell _{c}+1}N^{\frac{5\beta }{2}%
}\left\Vert S^{(1,k+\ell _{c}+1)}\omega _{N,H}^{(k+\ell
_{c}+1)}(t)\right\Vert _{L_{t_{k}}^{\infty }\left[ 0,T\right] L_{x,x^{\prime
}}^{2}}
\end{eqnarray}
\end{proposition}

\begin{proof}
See \S \ref{Sec:2.1}.
\end{proof}

The interaction part is addressed by following the method in \cite%
{Chen3DDerivation} which was inspired by \cite{TCNP5} of using $\ell
_{c}=\ln N$ to gain a negative power of $N$ from the power-of-$T$
coefficient in the above estimate. Then we can use the crude bound 
\begin{equation*}
\left\Vert \omega _{N,H}^{(k+\ell _{c}+1)}(t)\right\Vert \leq \left\Vert
\gamma _{N}^{(k+\ell _{c}+1)}(t)\right\Vert +\left\Vert \gamma _{H}^{(k+\ell
_{c}+1)}(t)\right\Vert
\end{equation*}%
that ignores the difference structure of $\omega _{N,H}$.

\begin{lemma}
For $\ell_c = \ln N$ and provided $T \lesssim E_0^{-2}$, 
\begin{equation}  \label{E:decomp-04}
\left\Vert S^{(1,k)}\text{\textsc{IP}}^{(k,\ell_c)}\right\Vert
_{L_{t_{k}}^{\infty }\left[ 0,T\right] L_{x,x^{\prime }}^{2}}\leqslant
CT^{1/2} (2E_0)^k N^{-2}
\end{equation}
\end{lemma}

\begin{proof}
By \eqref{E:decomp-05} and the energy bounds on $\gamma_N^{(k+\ell_c+1)}(t)$
and $\gamma_H^{(k+\ell_c+1)}(t)$, it suffices to show that 
\begin{equation*}
(4C E_0 T^{1/2})^{\ell_c+1} N^{\frac52\beta} \leq N^{-2}
\end{equation*}
We assume $T \lesssim E_0^{-2}$, specifically that $T$ is small enough so
that $4CE_0 T^{1/2} \leq e^{-5}$. Then 
\begin{equation*}
(4C E_0 T^{1/2})^{\ell_c+1} \leq e^{-5\ell_c} = e^{-5\ln N} = N^{-5}
\end{equation*}
\end{proof}

That is, the interaction part estimate can be made into $N^{-s}$ for any $s$%
, the limiting factor is solely the potential part which will get better
once one puts in the correlation functions as in \cite{C-H<1}.

Carrying out the sum in $k$ for the estimates in Proposition \ref%
{Thm:Estimates for FP,PP,IP in H1} gives us what we need in the master norm %
\eqref{norm:master capital}.

\begin{theorem}
\label{Thm:BBGKY Estimates in Masternorm} For $T$ and $Z$ such that $%
T\lesssim Z^{-2}$ and $Z\gtrsim E_{0}$, 
\begin{eqnarray*}
&&\Vert \Gamma _{N}(t)-\left\{ \left\vert \phi _{N}(t)\right\rangle
\left\langle \phi _{N}(t)\right\vert ^{\otimes k}\right\} \Vert
_{L_{[0,T]}^{\infty }H_{Z}^{1}} \\
&\leqslant &_{E_{0},Z}C\Vert \Gamma _{N}(0)-\left\{ \left\vert \phi
_{0}\right\rangle \left\langle \phi _{0}\right\vert ^{\otimes k}\right\}
\Vert _{H_{Z/2}^{1}}+CT^{1/2}N^{\frac{5}{2}\beta -1}
\end{eqnarray*}
\end{theorem}

\begin{proof}
See \S \ref{Sec:2.2}.
\end{proof}

\subsection{Proof of Proposition \protect\ref{Thm:Estimates for FP,PP,IP in
H1}\label{Sec:2.1}}

First of all, the summands inside each part can be grouped / combined
together further using the KM board game argument \cite{KM}, which is below,
to avoid a factorial factor.

\begin{lemma}[{\protect\cite[Lemma 2.1]{C-H2/3}}]
\label{lemma:Klainerman-MachedonBoardGameForBBGKY}For $j\geqslant 1$, one
can express 
\begin{equation*}
\int_{0}^{t_{k}}\cdots \int_{0}^{t_{k+j-1}}J_{N}^{(k,j)}(\underline{t}%
_{(k,j)})(f^{(k+j)})d\underline{t}_{(k,j)}
\end{equation*}%
as a sum of at most $2^{k+2j-2}$ terms of the form 
\begin{equation*}
\int_{D}J_{N}^{(k,j)}(\underline{t}_{(k,j)},\mu _{m})(f^{(k+j)})d\underline{t%
}_{(k,j)},
\end{equation*}%
or in other words, 
\begin{equation*}
\int_{0}^{t_{k}}\cdots \int_{0}^{t_{k+j-1}}J_{N}^{(k,j)}(\underline{t}%
_{(k,j)})(f^{(k+j)})d\underline{t}_{(k,j)}=\sum_{m}\int_{D}J_{N}^{(k,j)}(%
\underline{t}_{(k,j)},\mu _{m})(f^{(k+j)})d\underline{t}_{(k,j)}.
\end{equation*}%
Here $D\subset \lbrack 0,t_{k}]^{j}$, $\mu _{m}$ are a set of maps from $%
\{k+1,\ldots ,k+j\}$ to $\{1,\ldots ,k+j-1\}$ and $\mu _{m}(l)<l$ for all $%
l, $ and 
\begin{eqnarray*}
&&J_{N}^{(k,j)}(\underline{t}_{(k,j)},\mu _{m})(f^{(k+j)}) \\
&=&\left( U^{(k)}(t_{k}-t_{k+1})B_{N,\mu _{m}(k+1),k+1}\right) \left(
U^{(k+1)}(t_{k+1}-t_{k+2})B_{N,\mu _{m}(k+2),k+2}\right) \cdots \\
&&\cdots \left( U^{(k+j-1)}(t_{k+j-1}-t_{k+j})B_{N,\mu _{m}(k+j),k+j}\right)
(f^{(k+j)}).
\end{eqnarray*}
\end{lemma}

The counting $2^{k+2j-2}$ in Lemma \ref%
{lemma:Klainerman-MachedonBoardGameForBBGKY} is actually an easy upper bound
of a Catalan number.

\begin{lemma}[counting of KM reduced forms]
\label{lemma:Klainerman-MachedonCatalanNo}The number of mappings 
\begin{equation*}
\mu :\{k+1,\ldots ,k+j\}\rightarrow \{1,\ldots ,k+j-1\}
\end{equation*}%
satisfying $\mu (r)<r$ for each $k+1\leq r\leq k+j$ that are nondecreasing ($%
\mu (r)\leq \mu (r+1)$ for each $k+1\leq r\leq k+j-1$) is at most the
Catalan number 
\begin{equation}
\mathcal{C}(k,j)\equiv \binom{k+2j-2}{j} \leq 2^{k+2j-2}  \label{def:C(k,j)}
\end{equation}
\end{lemma}

\begin{proof}
We can associate to every reduced map $\mu$ a sequence $s$ 
\begin{equation*}
s(1) = \mu(k+1) \,, \quad s(2) = \mu(k+2)+1 \,, \quad \ldots, \quad s(j) =
\mu(k+j)+j-1
\end{equation*}
Note that $s$ is a (strictly) increasing subsequence of $\{1, \ldots,
k+2j-2\}$ of length $j$. Moreover, this process of converting from $\mu$ to $%
s$ is invertible: for any increasing subsequence of $\{1, \ldots, k+2j-2\}$
of length $j$, let $\mu$ be defined by 
\begin{equation*}
\mu(k+a) = s(a) - a+1 \,, \qquad \text{for }a = 1, \ldots, j
\end{equation*}
Since $s$ necessarily satisfies $s(a) \leq k+j+a-2$, it follows that $%
\mu(k+a) \leq k+j-2$ but this condition is not strong enough to guarantee
admissibility $\mu(k+a) \leq k+a-1$. Thus, the count of the number of
increasing subsequences of $\{1, \ldots, k+2j-2\}$ of length $j$, which is %
\eqref{def:C(k,j)}, is an over-count of the number of reduced admissible
maps $\mu$, but a useful upper bound.
\end{proof}

We can then estimate $J_{N}^{(k,j)}(\underline{t}_{(k,j)})(f^{(k+j)})$ via
the collapsing estimate in Lemma \ref{Lemma:H^1K-M Estimate}.

\begin{claim}
\label{Claim:IteratedCollapsingEstimates}For $j\geqslant 1$,%
\begin{align*}
\hspace{0.3in}&\hspace{-0.3in} \left\Vert \int_{0}^{t_{k}}\cdots
\int_{0}^{t_{k+j-1}}S^{(1,k)}J_{N}^{(k,j)}(\underline{t}_{(k,j)})(f^{(k+j)})d%
\underline{t}_{(k,j)}\right\Vert _{L_{t_{k}}^{\infty }\left[ 0,T\right]%
L_{x,x^{\prime }}^{2}} \\
&\leq 2^k (4CT^{1/2})^j \left\| S^{(1,k+j-1)} B_{N,1,k+j} f^{(k+j)}(t_{k+j})
\right\|_{L_{t_{k+j}}^2[0,T] L_{x,x^{\prime }}^2}
\end{align*}
\end{claim}

\begin{proof}
The proof follows the same steps usually used to estimate%
\begin{equation*}
\left\Vert \int_{0}^{t_{k}}\cdots
\int_{0}^{t_{k+j-1}}S^{(1,k-1)}B_{N,1,k}J_{N}^{(k,j)}(\underline{t}%
_{(k,j)})(f^{(k+j)})d\underline{t}_{(k,j)}\right\Vert _{L_{t_{k}}^{1}\left[
0,T\right] L_{x,x^{\prime }}^{2}}
\end{equation*}%
and is well-known by now. We include the proof for completeness. We start by
using Lemma \ref{lemma:Klainerman-MachedonBoardGameForBBGKY}, 
\begin{eqnarray*}
&&\left\Vert \int_{0}^{t_{k}}\cdots
\int_{0}^{t_{k+j-1}}S^{(1,k)}J_{N}^{(k,j)}(\underline{t}_{(k,j)})(f^{(k+j)})d%
\underline{t}_{(k,j)}\right\Vert _{L_{t_{k}}^{\infty }\left[ 0,T\right]
L_{x,x^{\prime }}^{2}} \\
&\leqslant &2^{k}4^{j}\left\Vert \int_{D}S^{(1,k)}J_{N}^{(k,j)}(\underline{t}%
_{(k,j)},\mu _{m})(f^{(k+j)})d\underline{t}_{(k,j)}\right\Vert
_{L_{t_{k}}^{\infty }L_{x,x^{\prime }}^{2}} \\
&\leqslant &2^{k}4^{j}\int_{\left[ 0,T\right] ^{j}}\left\Vert
S^{(1,k)}J_{N}^{(k,j)}(\underline{t}_{(k,j)},\mu _{m})(f^{(k+j)})\right\Vert
_{L_{x,x^{\prime }}^{2}}d\underline{t}_{(k,j)}
\end{eqnarray*}%
Cauchy-Schwarz at $d{t_{k+1}}$%
\begin{eqnarray*}
&\leqslant &2^{k}4^{j}T^{\frac{1}{2}}\int_{\left[ 0,T\right] ^{j-1}}d%
\underline{t}_{(k+1,j-1)} \\
&&\times \left\Vert S^{(1,k)}B_{N,\mu
_{m}(k+1),k+1}U^{(k+1)}(t_{k+1}-t_{k+2})...\right\Vert
_{L_{t_{k}}^{2}L_{x,x^{\prime }}^{2}}
\end{eqnarray*}%
Use Lemma \ref{Lemma:H^1K-M Estimate},%
\begin{eqnarray*}
&\leqslant &2^{k}4^{j}CT^{\frac{1}{2}}\int_{\left[ 0,T\right] ^{j-1}}d%
\underline{t}_{(k+1,j-1)} \\
&&\times \left\Vert S^{(1,k+1)}B_{N,\mu
_{m}(k+2),k+2}U^{(k+2)}(t_{k+2}-t_{k+3})...\right\Vert _{L_{x,x^{\prime
}}^{2}}
\end{eqnarray*}%
Repeating such a process gives%
\begin{equation*}
\leqslant 2^{k}4^{j}C^{j-1}T^{\frac{j-1}{2}}\int_{[0,T]}\left\Vert
S^{(1,k+j-1)}B_{N,\mu _{m}(k+j),k+j}(f^{(k+j)})\right\Vert _{L_{x,x^{\prime
}}^{2}}dt_{k+j}
\end{equation*}%
By symmetry, 
\begin{equation*}
=2^{k}\left( 4CT^{\frac{1}{2}}\right) ^{j-1}\int_{[0,T]}\left\Vert
S^{(1,k+j-1)}B_{N,1,k+j}(f^{(k+j)})\right\Vert _{L_{x,x^{\prime
}}^{2}}dt_{k+j}
\end{equation*}%
Applying Cauchy-Schwarz in time once more yields the claim.
\end{proof}

Starting with the formulae for $\text{FP}^{(k,\ell _{c})}$, D$\text{P}%
^{(k,\ell _{c})}$, and $\text{IP}^{(k,\ell _{c})}$, we apply Lemma \ref%
{lemma:Klainerman-MachedonBoardGameForBBGKY} using the bound in Lemma \ref%
{lemma:Klainerman-MachedonCatalanNo} to reduce the number of Duhamel terms,
and apply the estimate in Claim \ref{Claim:IteratedCollapsingEstimates} for
each term. This provides preliminary estimates for the three parts in the
expansion of $\omega _{N,H}^{(k)}$.

Specifically, for the free part, this yields 
\begin{eqnarray*}
&&\left\Vert S^{(1,k)}\text{\textsc{FP}}^{(k,\ell _{c})}\right\Vert
_{L_{t_{k}}^{\infty }\left[ 0,T\right] L_{x,x^{\prime }}^{2}} \\
&\leq &\Vert S^{(1,k)}f_{\text{FP}}^{(k,0)}(t_{k})\Vert _{L_{t_{k}}^{\infty
}[0,T]L_{x,x^{\prime }}^{2}} \\
&&+2^{k}\sum_{j=1}^{\ell _{c}}\left( 4CT^{\frac{1}{2}}\right) ^{j}\Vert
S^{(1,k+j-1)}B_{N,1,j+k}f_{\text{FP}}^{(k,j)}(t_{k+j})\Vert
_{L_{t_{k+j}}^{2}[0,T]L_{x,x^{\prime }}^{2}}
\end{eqnarray*}%
Plugging in $f_{\text{FP}}^{(k,j)}$ and applying the Klainerman-Machedon
trilinear estimate (Lemma \ref{Lemma:H^1K-M Estimate}), 
\begin{equation*}
\leq \Vert S^{(1,k)}\omega _{N,H}^{(k)}(0)\Vert _{L_{t_{k}}^{\infty
}[0,T]L_{x,x^{\prime }}^{2}}+2^{k}\sum_{j=1}^{\ell _{c}}\left( 4CT^{\frac{1}{%
2}}\right) ^{j}\Vert S^{(1,k+j)}\omega _{N,H}^{(k+j)}(0)\Vert
_{L_{x,x^{\prime }}^{2}}
\end{equation*}%
which completes the proof for the free part in Proposition \ref%
{Thm:Estimates for FP,PP,IP in H1}.

For the driving part, this yields%
\begin{eqnarray}
&&\left\Vert S^{(1,k)}\text{\textsc{DP}}^{(k,\ell _{c})}\right\Vert
_{L_{t_{k}}^{\infty }\left[ 0,T\right] L_{x,x^{\prime }}^{2}}
\label{E:PP-intermediate-01} \\
&\leq &\Vert S^{(1,k)}f_{\text{DP}}^{(k,0)}(t_{k})\Vert _{L_{t_{k}}^{\infty
}[0,T]L_{x,x^{\prime }}^{2}}  \notag \\
&&+2^{k}\sum_{j=1}^{\ell _{c}}\left( 4CT^{\frac{1}{2}}\right) ^{j}\Vert
S^{(1,k+j-1)}B_{N,1,k+j}f_{\text{DP}}^{(k,j)}(t_{k+j})\Vert
_{L_{t_{k+j}}^{2}[0,T]L_{x,x^{\prime }}^{2}}  \notag
\end{eqnarray}

For the interaction part, this yields 
\begin{eqnarray}
&&\left\Vert S^{(1,k)}\text{\textsc{IP}}^{(k,\ell_c)}\right\Vert
_{L_{t_{k}}^{\infty }\left[ 0,T\right] L_{x,x^{\prime }}^{2}}
\label{estimate:IP estimate after iteration} \\
&\leqslant & 2^k \left( 4CT^{\frac{1}{2}}\right) ^{\ell_c+1}\left\Vert
S^{(1,k+\ell_c)}B_{N,1,k+\ell_c+1}\omega
_{N,H}^{(k+\ell_c+1)}(t_{k+\ell_c+1})\right\Vert
_{L_{t_{k+\ell_c+1}^2}L_{x,x^{\prime }}^{2}}  \notag
\end{eqnarray}%
We continue the estimates of the driving part and the interaction part
separately below.

\subsubsection{Estimate for the Driving Part}

We complete the bound of the right side of \eqref{E:PP-intermediate-01}.
Using the $X_{\frac{1}{2}+}^{(k+j)}\hookrightarrow L_{t_{k+j}}^{\infty
}[0,T]L_{x,x^{\prime }}^{2}$ embedding, 
\begin{equation*}
\Vert f_{\text{DP}}^{(k,0)}(t_{k})\Vert _{L_{t_{k}}^{\infty
}[0,T]L_{x,x^{\prime }}^{2}}\leq C\Vert \theta (t_{k})S^{(1,k)}f_{\text{DP}%
}^{(k,0)}(t_{k})\Vert _{X_{\frac{1}{2}+}^{(k)}}
\end{equation*}%
where $\theta (t)$ is a smooth cutoff in time such that $\theta (t)=1$ on $%
[0,T]$. For $1\leq j\leq \ell _{c}$, by Lemma \ref{Lemma:KMEstimateInWithX_b}
(a version of the Klainerman-Machedon trilinear estimate with $X$-norm on
the right side) 
\begin{eqnarray*}
&&\Vert S^{(1,k+j-1)}B_{N,1,k+j}f_{\text{DP}}^{(k,j)}(t_{k+j})\Vert
_{L_{t_{k+j}}^{2}[0,T]L_{x,x^{\prime }}^{2}} \\
&\leq &C\Vert \theta (t_{k+j})S^{(1,k+j)}f_{\text{DP}}^{(k,j)}(t_{k+j})\Vert
_{X_{\frac{1}{2}+}^{(k+j)}}
\end{eqnarray*}%
Thus to complete the bound of \eqref{E:PP-intermediate-01}, it remains to
estimate for $0\leq j\leq \ell _{c}$, 
\begin{equation}
\Vert \theta (t_{k+j})S^{(1,k+j)}f_{\text{DP}}^{(k,j)}(t_{k+j})\Vert _{X_{%
\frac{1}{2}+}^{(k+j)}}  \label{E:PP-intermediate-02}
\end{equation}%
Referring to the definition \eqref{eqn:def of f inside PP} of $f_{\text{DP}%
}^{(k,j)}$, insert $\tilde{\theta}(t_{k+j+1})$ inside the integrand, where $%
\tilde{\theta}(t)$ is a smooth cutoff in time such that $\tilde{\theta}(t)=1$
on the support of $\theta (t)$. Applying Claim \ref{Claim:b to b-1}, 
\begin{equation*}
\leq C\Vert \tilde{\theta}(t_{k+j+1})S^{(1,k+j)}V_{N}^{(k+j)}\gamma
_{N}^{(k+j)}(t_{k+j+1})\Vert _{X_{-\frac{1}{2}+}^{(k+j)}}
\end{equation*}%
By dual Strichartz (Lemma \ref{Lemma:PP Estimate in Strichartz form}) we
complete the bound of \eqref{E:PP-intermediate-02} by 
\begin{eqnarray*}
&\leq &CN^{\frac{5}{2}\beta -1}(k+j)^{2}\Vert S^{(1,k+j)}\gamma
_{N}^{(k+j)}(t_{k+j+1})\Vert _{L_{t_{k+j+1}}^{2}L_{x,x^{\prime }}^{2}} \\
&\leq &CT^{1/2}N^{\frac{5}{2}\beta -1}(k+j)^{2}E_{0}^{k+j}
\end{eqnarray*}%
where, in the last step, we appealed to the energy bound and the $(k+j)^{2}$
factor came from the expansion of $V_{N}^{(k+j)}$ into component terms.
Inserting this to bound of \eqref{E:PP-intermediate-02} into the right side
of \eqref{E:PP-intermediate-01}, 
\begin{align*}
\Vert S^{(1,k)}\text{DP}^{(k,\ell _{c})}\Vert _{L_{t_{k}}^{\infty
}[0,T]L_{x,x^{\prime }}^{2}}& \leq CT^{1/2}N^{\frac{5}{2}\beta
-1}2^{k}\sum_{j=0}^{\ell _{c}}(k+j)^{2}(4CT^{1/2})^{j}E_{0}^{k+j} \\
& \leq CT^{1/2}N^{\frac{5}{2}\beta -1}(2E_{0})^{k}k^{2}
\end{align*}%
provided $T$ is small enough so that $4CT^{1/2}E_{0}\leq \frac{1}{2}$, which
completes the bound for the driving part in Proposition \ref{Thm:Estimates
for FP,PP,IP in H1}.

\subsubsection{Estimate for the Interaction Part}

From \eqref{estimate:IP estimate after iteration}, we see that it remains to
bound 
\begin{equation}
\left\Vert S^{(1,k+\ell _{c})}B_{N,1,k+\ell _{c}+1}\omega _{N,H}^{(k+\ell
_{c}+1)}(t_{k+\ell _{c}+1})\right\Vert _{L_{t_{k+\ell
_{c}+1}^{2}}L_{x,x^{\prime }}^{2}}  \label{E:IP-intermediate-01}
\end{equation}%
Take the crude estimate that burns derivatives and gains bad powers of $N$: 
\begin{eqnarray*}
&&\left\Vert S^{(1,k+\ell _{c})}B_{N,1,k+\ell _{c}+1}\omega _{N,H}^{(k+\ell
_{c}+1)}(t_{k+\ell _{c}+1})\right\Vert _{L_{t_{k+\ell
_{c}+1}^{2}}L_{x,x^{\prime }}^{2}}^{2} \\
&\leqslant &CTN^{2\beta }\left\Vert V_{N}^{\prime }\right\Vert
_{L^{2}}^{2}\left\Vert S^{(1,k+\ell _{c})}\omega _{N,H}^{(k+\ell _{c}+1)}(t%
\mathbf{,x}_{k},x_{k+1},\mathbf{x}_{k}^{\prime },x_{k+1})\right\Vert
_{L_{t}^{\infty }\left[ 0,T\right] L_{x,x^{\prime }}^{2}}^{2}
\end{eqnarray*}%
and use the trace theorem, 
\begin{equation*}
\leqslant CTN^{5\beta }\left\Vert V^{\prime }\right\Vert
_{L^{2}}^{2}\left\Vert S^{(1,k+\ell _{c}+1)}\omega _{N,H}^{(k+\ell
_{c}+1)}\right\Vert _{L_{t}^{\infty }\left[ 0,T\right] L_{x,x^{\prime
}}^{2}}^{2}.
\end{equation*}%
Inserting this estimate of \eqref{E:IP-intermediate-01} into %
\eqref{estimate:IP estimate after iteration}, 
\begin{eqnarray*}
&&\left\Vert S^{(1,k)}\text{\textsc{IP}}^{(k,\ell _{c})}\right\Vert
_{L_{t_{k}}^{\infty }\left[ 0,T\right] L_{x,x^{\prime }}^{2}} \\
&\leqslant &2^{k}T^{1/2}(4CT^{1/2})^{\ell _{c}+1}N^{\frac{5\beta }{2}%
}\left\Vert S^{(1,k+\ell _{c}+1)}\omega _{N}^{(k+\ell _{c}+1)}(t)\right\Vert
_{L_{t_{k}}^{\infty }\left[ 0,T\right] L_{x,x^{\prime }}^{2}}
\end{eqnarray*}

\subsection{Summing in $k$ / Proof of Theorem \protect\ref{Thm:BBGKY
Estimates in Masternorm}\label{Sec:2.2}}

Using the definition of the master norm \eqref{norm:master capital} and the
decomposition \eqref{E:decomp-01}%
\begin{eqnarray*}
&&\Vert \Gamma _{N}(t)-\left\{ \left\vert \phi _{N}(t)\right\rangle
\left\langle \phi _{N}(t)\right\vert ^{\otimes k}\right\} \Vert
_{L_{[0,T]}^{\infty }H_{Z}^{1}} \\
&\leqslant &\sum_{k=0}^{\infty }Z^{-k}(\Vert S^{(1,k)}\text{FP}^{(k,\ell
_{c})}\Vert _{L_{t_{k}}^{\infty }[0,T]L_{x,x^{\prime }}^{2}} \\
&&+\Vert S^{(1,k)}\text{DP}^{(k,\ell _{c})}\Vert _{L_{t_{k}}^{\infty
}[0,T]L_{x,x^{\prime }}^{2}}+\Vert S^{(1,k)}\text{IP}^{(k,\ell _{c})}\Vert
_{L_{t_{k}}^{\infty }[0,T]L_{x,x^{\prime }}^{2}})
\end{eqnarray*}%
Applying the bounds on each component in \eqref{E:decomp-02}, %
\eqref{E:decomp-03}, \eqref{E:decomp-04}, we obtain%
\begin{eqnarray*}
&&\Vert \Gamma _{N}(t)-\left\{ \left\vert \phi _{N}(t)\right\rangle
\left\langle \phi _{N}(t)\right\vert ^{\otimes k}\right\} \Vert
_{L_{[0,T]}^{\infty }H_{Z}^{1}} \\
&\leqslant &\sum_{k=0}^{\infty }\sum_{j=0}^{\ell
_{c}}(2Z^{-1})^{k}(4CT^{1/2})^{j}\Vert S^{(1,k+j)}\omega
_{N,H}^{(k+j)}(0)\Vert _{L_{x,x^{\prime }}^{2}} \\
&&+CT^{1/2}N^{\frac{5}{2}\beta -1}\sum_{k=0}^{\infty }(2E_{0}Z^{-1})^{k}k^{2}
\end{eqnarray*}%
In the double sum, changing $(k,j)$ to $(m,j)$ where $m=k+j$, and using the
discrete Fubini that $\sum_{k=0}^{\infty }\sum_{j=0}^{\ell
_{c}}=\sum_{j=0}^{\ell _{c}}\sum_{k=0}^{\infty }=\sum_{j=0}^{\ell
_{c}}\sum_{m=j}^{\infty }=\sum_{m=0}^{\infty }\sum_{j=0}^{\min (m,\ell
_{c})} $, we get%
\begin{eqnarray*}
&&\Vert \Gamma _{N}(t)-\left\{ \left\vert \phi _{N}(t)\right\rangle
\left\langle \phi _{N}(t)\right\vert ^{\otimes k}\right\} \Vert
_{L_{[0,T]}^{\infty }H_{Z}^{1}} \\
&\leqslant &\sum_{m=0}^{\infty }\sum_{j=0}^{\min (m,\ell
_{c})}(2Z^{-1})^{m}(2CZT^{1/2})^{j}\Vert S^{(1,m)}\omega
_{N,H}^{(m)}(0)\Vert _{L_{x,x^{\prime }}^{2}} \\
&&+CT^{1/2}N^{\frac{5}{2}\beta -1}\sum_{k=0}^{\infty }(2E_{0}Z^{-1})^{k}k^{2}
\end{eqnarray*}%
Provided $Z\gtrsim E_{0}$ and $T\lesssim Z^{-2}$, we can carry out the $j$
and $k$ sums. This completes the proof of Theorem \ref{Thm:BBGKY Estimates
in Masternorm}.

\section{Comparing H-NLS and NLS\label{sec:Hnls vs NLS}}

In this section, we give the proof of Theorem \ref{Thm:biScattering} which
will be concluded after Propositions \ref{P:diff-control} and \ref%
{P:diff-control-ep}. The estimate \eqref{estimate:key-2 for almost optimal}
in the introduction then follow.

We need the atomic $U$ spaces introduced by Koch \& Tataru \cite{KT-UV1,
KT-UV2} and the $V$ spaces of bounded $p$-variation of Wiener \cite{Wie-V}.
Their properties have been further elaborated in Hadac, Herr, \& Koch \cite%
{HHK} and Koch, Tataru, \& Visan \cite{KTV}. Here, following \cite[%
Definition 2.1 and Definition 2.3]{HHK} (see the slight change in the
erratum for that paper), we define $U^{p}(I;H)$ and $V^{p}(I;H)$, where $I=%
\left[ T_{1},T_{2}\right) \subset \mathbb{R}$ is a time interval and $H$ is
a Hilbert space (in $x$) below.

Let $\mathcal{Z}$ be the set of all finite partitions $%
T_{1}=t_{0}<t_{1}<...<t_{K}\leqslant T_{2}$ of $I\ $and let us use the
convention that $v(T_{2})=0$ for all functions $v:I\rightarrow H$.

\begin{definition}
Let $p\in \left[ 1,\infty \right) .$We call a function $a:I\rightarrow H$ a $%
U^{p}$-atom if it takes the form $a=\sum_{k=1}^{K}\mathbf{1}_{\left[
t_{k-1},t_{k}\right) }\phi _{k-1}$ where $\left\{ t_{k}\right\} \in $ $%
\mathcal{Z}$ and $\left\{ \phi _{k}\right\} \subset H$ with $%
\sum_{k=0}^{K-1}\left\Vert \phi _{k}\right\Vert _{H}^{p}=1$. The atomic
space $U^{p}(I;H)\subset L^{\infty }(I;H)$ is the space of functions $%
u:I\rightarrow H$ given the norm:%
\begin{equation*}
\left\Vert u\right\Vert _{U^{p}}=\inf \left\{ \sum_{j=1}^{\infty }\left\vert
\lambda _{j}\right\vert :u=\sum_{j=1}^{\infty }\lambda _{j}a,\text{ }\lambda
_{j}\in \mathbb{C}\text{, }a_{j}\text{ is a }U^{p}\text{-atom for all }%
j\right\} .
\end{equation*}
\end{definition}

\begin{definition}
Let $p\in \left[ 1,\infty \right) .$The space $V^{p}(I;H)$ is the space of
of all functions $v:I\rightarrow H$ such that 
\begin{equation*}
\left\Vert v\right\Vert _{V^{p}}=\sup_{\left\{ t_{k}\right\} \in \mathcal{Z}%
}\left( \sum_{j=1}^{\infty }\left\Vert v(t_{k})-v(t_{k-1})\right\Vert
_{H}^{p}\right) ^{\frac{1}{p}}<+\infty
\end{equation*}%
and the space $V_{\text{rc}}^{p}(I;H)$ denotes the closed subspace of all
right-continuous functions $v:I\rightarrow H$ such that $v(T_{1})=0.$
\end{definition}

We have, for $1\leq p<q<\infty $ (see Proposition 2.4, Corollary 2.6 in \cite%
{HHK}) the continuous embeddings 
\begin{equation}
U^{p}\hookrightarrow V_{\text{rc}}^{p}\hookrightarrow U^{q}\hookrightarrow
L^{\infty }  \label{E:UV-basic-embeddings}
\end{equation}%
We in fact work exclusively with the variants $U_{\Delta }^{p}L_{x}^{2}$, $%
V_{\Delta }^{p}L_{x}^{2}$ defined as the $U^{p}L_{x}^{2}$ and $%
V^{p}L_{x}^{2} $ norms, respectively, after pulling-back by the linear flow $%
e^{it\Delta }$ (as in \cite[Definition 2.15]{HHK}), and will denote the
restriction of such norms to a time subinterval $I$ as $U_{I,\Delta
}^{p}L_{x}^{2}$ and $V_{I,\Delta }^{p}L_{x}^{2}$.

It is immediate from the definition of the $U^p_{I,\Delta}L_x^2$ norm that
for any $1\leq p <\infty$, 
\begin{equation*}
\| e^{it\Delta} \phi \|_{U^p_{I,\Delta}L_x^2} \leq \|\phi\|_{L_x^2}
\end{equation*}
From \cite[Theorem 2.8, Proposition 2.10]{HHK}, we have the duality
relationship 
\begin{equation}  \label{E:UV-dual}
\left\| \int_0^t e^{i(t-t^{\prime })\Delta} f(t^{\prime }) \, dt^{\prime
}\right\|_{U^2_{I,\Delta}L_x^2} = \sup_{\substack{ g\in V^2_{I,\Delta}L_x^2 
\\ \|g\|_{V^2_{I,\Delta}L_x^2}\leq 1}} \left| \int_I \int_x f(x,t) \, g(x,t)
\, dx \,dt \right|
\end{equation}
which is key to estimating Duhamel terms.

It follows from \cite[Proposition 2.19]{HHK} that the Strichartz estimates
imply 
\begin{equation}
\Vert u\Vert _{L_{I}^{q}L_{x}^{r}}\lesssim \Vert u\Vert _{U_{I,\Delta
}^{q}L_{x}^{2}}  \label{E:Str-U}
\end{equation}%
for admissible $(q,r)$: 
\begin{equation*}
\frac{2}{q}+\frac{3}{r}=\frac{3}{2}\,,\qquad 2\leq q<\infty \,,\quad 2\leq
r\leq 6
\end{equation*}%
where we note that the $q$ exponent appears on both the left and right. From %
\eqref{E:UV-basic-embeddings}, the larger the $q$, the smaller the right
side (the better the resulting bound) in \eqref{E:Str-U}.

Also, from \cite[Proposition 2.20]{HHK}, we have the following property as a
substitute for the failure of the $V^{2}\hookrightarrow U^{2}$ embedding
(compare \eqref{E:UV-basic-embeddings}). If $T$ is a bilinear operator
satisfying 
\begin{equation}
\Vert T(u_{1},u_{2})\Vert _{L_{I}^{2}L_{x}^{2}}\leq C\Vert u_{1}\Vert
_{U_{I,\Delta }^{q}L_{x}^{2}}\Vert u_{2}\Vert _{U_{I,\Delta }^{q}L_{x}^{2}}
\label{E:UtoV01}
\end{equation}%
for some $q>2$ and 
\begin{equation}
\Vert T(u_{1},u_{2})\Vert _{L_{I}^{2}L_{x}^{2}}\leq C_{2}\Vert u_{1}\Vert
_{U_{I,\Delta }^{2}L_{x}^{2}}\Vert u_{2}\Vert _{U_{I,\Delta }^{2}L_{x}^{2}}
\label{E:UtoV02}
\end{equation}%
then it follows that 
\begin{equation}
\Vert T(u_{1},u_{2})\Vert _{L_{I}^{2}L_{x}^{2}}\leq C_{2}\left( \log \frac{C%
}{C_{2}}+1\right) \Vert u_{1}\Vert _{U_{I,\Delta }^{2}L_{x}^{2}}\Vert
u_{2}\Vert _{V_{I,\Delta }^{2}L_{x}^{2}}  \label{E:UtoV03}
\end{equation}%
To present an application that we need below, first note that the following
bilinear Strichartz estimate holds.

\begin{lemma}[blinear Strichartz \protect\cite{B1}]
\label{L:bil-Str}For $x\in \mathbb{R}^{d}$,%
\begin{equation*}
\Vert P_{M_{1}}e^{it\Delta }\phi _{1}\,\overline{P_{M_{2}}e^{it\Delta }\phi
_{2}}\Vert _{L_{[0,1]}^{2}L_{x}^{2}}\lesssim \left( \frac{\min
(M_{1},M_{2})^{d-1}}{\max (M_{1},M_{2})}\right) ^{1/2}\Vert P_{M_{1}}\phi
_{1}\Vert _{L_{x}^{2}}\Vert P_{M_{2}}\phi _{2}\Vert _{L_{x}^{2}}
\end{equation*}%
which is, in $U$-$V$ notation,%
\begin{equation}
\Vert P_{M_{1}}\phi _{1}\,\overline{P_{M_{2}}\phi _{2}}\Vert
_{L_{[0,1]}^{2}L_{x}^{2}}\lesssim \left( \frac{\min (M_{1},M_{2})^{d-1}}{%
\max (M_{1},M_{2})}\right) ^{1/2}\Vert P_{M_{1}}\phi _{1}\Vert _{U_{I,\Delta
}^{2}L_{x}^{2}}\Vert P_{M_{2}}\phi _{2}\Vert _{U_{I,\Delta }^{2}L_{x}^{2}}
\label{E:bil-UV}
\end{equation}
\end{lemma}

Lemma \ref{L:bil-Str} fits the template \eqref{E:UtoV02} with $%
T(u_{1},u_{2})=u_{1}u_{2}$, $u_{1}=P_{M_{1}}\phi _{1}$, $u_{2}=P_{M_{2}}\phi
_{2}$, and $C_{2}=\frac{\min (M_{1},M_{2})}{\max (M_{1},M_{2})^{1/2}}$.
However, by H\"{o}lder, Sobolev and Strichartz estimates, we have 
\begin{eqnarray*}
\Vert P_{M_{1}}\phi _{1}P_{M_{2}}\phi _{2}\Vert _{L_{I}^{2}L_{x}^{2}}
&\lesssim &\Vert P_{M_{1}}\phi _{1}\Vert _{L_{I}^{4}L_{x}^{4}}\Vert
P_{M_{1}}\phi _{2}\Vert _{L_{I}^{4}L_{x}^{4}} \\
&\lesssim &M_{1}^{1/4}M_{2}^{1/4}\Vert \phi _{1}\Vert _{U_{I,\Delta
}^{4}L_{x}^{2}}\Vert \phi _{2}\Vert _{U_{I,\Delta }^{4}L_{x}^{2}}
\end{eqnarray*}%
which fits the template of \eqref{E:UtoV01} with $q=4$ and $%
C=M_{1}^{1/4}M_{2}^{1/4}$. The conclusion \eqref{E:UtoV03} reads 
\begin{eqnarray}
&&\Vert P_{M_{1}}\phi _{1}P_{M_{2}}\phi _{2}\Vert _{L_{I}^{2}L_{x}^{2}}
\label{E:bil-UV02} \\
&\lesssim &\frac{\min (M_{1},M_{2})}{\max (M_{1},M_{2})^{1/2}}\left( 1+\log 
\frac{\max (M_{1},M_{2})}{\min (M_{1},M_{2})}\right)  \notag \\
&&\Vert P_{M_{1}}\phi _{1}\Vert _{U_{I,\Delta }^{2}L_{x}^{2}}\Vert
P_{M_{2}}\phi _{2}\Vert _{V_{I,\Delta }^{2}L_{x}^{2}}  \notag
\end{eqnarray}%
where in fact the position of the $U_{\Delta }^{2}$ and $V_{\Delta }^{2}$
norms on the right can be switched. The result is that we have been able to
take \eqref{E:bil-UV} and upgrade one of the norms on the right side to $%
V^{2}$ at the expense a logarithmic loss.


After this background, we now proceed with the proof of Theorem \ref%
{Thm:biScattering}. Recall $\phi $ and $\phi _{N}$ are the solutions to (\ref%
{NLS:opening cubic}) and (\ref{NLS:Hartree}) and let 
\begin{equation*}
\tilde{\phi}=\phi _{N}-\phi .
\end{equation*}

It follows from energy conservation and classical well-posedness theory in
the Strichartz spaces that (\ref{NLS:opening cubic}) and (\ref{NLS:Hartree})
in the $\mathbb{R}^{3}$ defocusing case satisfy the global in time bounds 
\begin{equation*}
\Vert \phi \Vert _{L_{t}^{\infty }H_{x}^{1}}\leq C_{1}\,,\qquad \Vert \phi
_{N}\Vert _{L_{t}^{\infty }H_{x}^{1}}\leq C_{1}
\end{equation*}%
where the constant $C_{1}$ depends on the size of the initial data in $H^{1}$%
. The following theorem on scattering was obtained for NLS by Ginibre \&
Velo \cite{GV} using a Morawetz estimate of Lin \& Strauss \cite{LS}. An
alternate proof using an interaction Morawetz was given by Colliander, Keel,
Staffilani, Takaoka, \& Tao \cite{CKSTT-sub}. A version in the focusing
setting by Duyckaerts, Holmer, Roudenko \cite{DHR} was obtained using the
concentration compactness and virial rigidity method of Kenig \& Merle \cite%
{KMerle}. The corresponding Hartree result was obtained by Miao, Xu, \& Zhao 
\cite{MXZ}.

\begin{theorem}[scattering]
\label{T:scat} Defocusing cubic NLS (\ref{NLS:opening cubic}) and defocusing
cubic H-NLS (\ref{NLS:Hartree}) in $\mathbb{R}^{3}$ both scatter in $H^{1}$.
In particular, for $H^{1}$ data, we have the global-in-time bounds 
\begin{equation*}
\Vert \langle \nabla \rangle \phi \Vert _{L_{t}^{2}L_{x}^{6}}\lesssim
1\,,\qquad \Vert \langle \nabla \rangle \phi _{N}\Vert
_{L_{t}^{2}L_{x}^{6}}\lesssim 1
\end{equation*}
\end{theorem}

Bounds on other Strichartz norms can be obtained by interpolation. As a
corollary, we have that there exists a \emph{finite} partition of the time
interval $[0,+\infty )$ 
\begin{equation*}
0=t_{0}<t_{1}<\cdots <t_{J}=\infty
\end{equation*}%
such that on each subinterval $I=[t_{j},t_{j+1})$ for $0\leq j\leq J$, there
holds 
\begin{equation}
\Vert \langle \nabla \rangle \phi \Vert _{L_{I}^{2}L_{x}^{6}}\leq \delta
\,,\qquad \Vert \langle \nabla \rangle \phi _{N}\Vert
_{L_{I}^{2}L_{x}^{6}}\leq \delta  \label{E:Strichartz-small}
\end{equation}

\begin{corollary}
\label{C:U2control} If $\delta>0$ is chosen small\footnote{%
The proof shows that $\delta\lesssim \langle C_1 \rangle^{-1/3}$ suffices}
in terms of $C_1$, then for each interval $I$ on which %
\eqref{E:Strichartz-small} holds, we have 
\begin{equation}  \label{E:U2control}
\Vert \phi \Vert _{U^2_{I,\Delta}H_x^1} \leq 2C_{1}\,,\qquad \Vert \phi
_{N}\Vert _{U^2_{I,\Delta}H_x^1} \leq 2C_{1}
\end{equation}
\end{corollary}

\begin{proof}
The argument for $\phi $ (NLS) and $\phi _{N}$ (HNLS) is similar, so we will
just write it for $\phi $. On $I=[t_{\ast },t^{\ast }]$, we have 
\begin{equation*}
\phi (t)=e^{i(t-t_{\ast })\Delta }\phi (t_{\ast })+\int_{t_{\ast
}}^{t}e^{i(t-t^{\prime })\Delta }\left\vert \phi (t^{\prime })\right\vert
^{2}\phi (t^{\prime })\,dt^{\prime }
\end{equation*}%
By \eqref{E:UV-dual}, 
\begin{equation}
\Vert \phi \Vert _{U_{I,\Delta }^{2}H_{x}^{1}}\lesssim \Vert \phi (t_{\ast
})\Vert _{H_{x}^{1}}+\sup_{\substack{ g\in V_{I,\Delta }^{2}L_{x}^{2}  \\ %
\Vert g\Vert _{V_{I,\Delta }^{2}L_{x}^{2}}\leq 1}}\left\vert
\int_{I}\int_{x}\langle \nabla \rangle (|\phi |^{2}\phi
)\;g\,dx\,dt\right\vert  \label{E:UV-stuff-01}
\end{equation}%
For a particular $g$, we estimate as 
\begin{equation*}
\left\vert \int_{I}\int_{x}\langle \nabla \rangle (|\phi |^{2}\phi
)\;g\,dx\,dt\right\vert \lesssim \Vert \langle \nabla \rangle \phi \Vert
_{L_{I}^{2}L_{x}^{6}}\Vert \phi \Vert _{L_{I}^{4}L_{x}^{3}}\Vert \phi \Vert
_{L_{I}^{\infty }L_{x}^{6}}\Vert g\Vert _{L_{I}^{4}L_{x}^{3}}
\end{equation*}%
By H\"{o}lder interpolation 
\begin{equation*}
\Vert \phi \Vert _{L_{I}^{4}L_{x}^{3}}\leq \Vert \phi \Vert
_{L_{I}^{2}L_{x}^{6}}^{1/2}\Vert \phi \Vert _{L_{x}^{\infty
}L_{x}^{2}}^{1/2}\leq \delta ^{1/2}C_{1}^{1/2}
\end{equation*}%
By Sobolev embedding, $\Vert \phi \Vert _{L_{I}^{\infty }L_{x}^{6}}\lesssim
C_{1}$. And by \eqref{E:Str-U} and \eqref{E:UV-basic-embeddings}, 
\begin{equation*}
\Vert g\Vert _{L_{I}^{4}L_{x}^{3}}\lesssim \Vert g\Vert _{U_{I,\Delta
}^{4}L_{x}^{2}}\lesssim \Vert g\Vert _{V_{I,\Delta }^{2}L_{x}^{2}}
\end{equation*}%
Inserting these above, we obtain 
\begin{equation*}
\left\vert \int_{I}\int_{x}\langle \nabla \rangle (|\phi |^{2}\phi
)\;g\,dx\,dt\right\vert \lesssim \delta ^{3/2}C_{1}^{3/2}\Vert g\Vert
_{V_{I,\Delta }^{2}L_{x}^{2}}
\end{equation*}%
By \eqref{E:UV-stuff-01}, and the fact that $\Vert \phi (t_{\ast })\Vert
_{H_{x}^{1}}\leq C_{1}$, we obtain the result.
\end{proof}

Now we will show that on each time interval $I$ in the finite partition of $%
0\leq t <+\infty$, we obtain a bound on $\tilde \phi$ in terms of the
initial difference for that subinterval.

\begin{proposition}
\label{P:diff-control} Suppose that on a time interval $I$ the solutions to (%
\ref{NLS:opening cubic}) and (\ref{NLS:Hartree}) satisfy 
\begin{equation*}
\Vert \phi \Vert _{U_{I,\Delta }^{2}H_{x}^{1}}\leq 2C_{1}\,,\qquad \Vert
\phi _{N}\Vert _{U_{I,\Delta }^{2}H_{x}^{1}}\leq 2C_{1}
\end{equation*}%
for some constant $C_{1}$ and 
\begin{equation*}
\Vert \langle \nabla \rangle \phi \Vert _{L_{I}^{2}L_{x}^{6}}\leq \delta
\,,\qquad \Vert \langle \nabla \rangle \phi _{N}\Vert
_{L_{I}^{2}L_{x}^{6}}\leq \delta
\end{equation*}%
Consider the difference 
\begin{equation*}
\tilde{\phi}(t)=\phi _{N}(t)-\phi (t)
\end{equation*}%
with initial condition $\tilde{\phi}_{0}=(\phi _{N})_{0}-\phi _{0}$ for the
time interval $I$.

Provided $\delta >0$ is chosen small, $\Vert \tilde{\phi}_{0}\Vert
_{H_{x}^{1}}$ is sufficiently small, and $N$ is sufficiently large (all of
these thresholds are expressed in terms of $C_{1}$ only), then we have 
\begin{equation}  \label{E:diff-control}
\Vert \tilde{\phi}\Vert _{U_{I,\Delta }^{2}H_{x}^{1}}\lesssim \Vert \tilde{%
\phi}_{0}\Vert _{H_{x}^{1}}+(\log N)^{7}N^{-\beta }C_{1}^{3}
\end{equation}
\end{proposition}

\begin{remark}
This result just fails by a logarithm to obtain the optimal $N^{-\beta }$
rate at $1$ derivative of regularity. With more delicate arguments, we can
indeed reduce the power on the $\log N$ factor, although we do not see a way
to completely eliminate the $\log N$ factor.
\end{remark}

\begin{proof}
Plug in $\phi _{N}=\phi +\tilde{\phi}$ into (\ref{NLS:Hartree}), and using
that $\phi $ solves (\ref{NLS:opening cubic}) to simplify, we obtain that $%
\tilde{\phi}$ solves 
\begin{equation*}
0=i\partial _{t}\tilde{\phi}+\Delta \tilde{\phi}-(V_{N}\ast |\phi +\tilde{%
\phi}|^{2})(\phi +\tilde{\phi})+|\phi |^{2}\phi
\end{equation*}%
Adopting the shorthand, 
\begin{equation*}
W_{N}(x)=N^{3\beta }V(N^{\beta }x)-b_{0}\delta (x)
\end{equation*}%
we expand the nonlinearity 
\begin{align*}
0& =i\partial _{t}\tilde{\phi}+\Delta \tilde{\phi} & & \\
& \qquad -(W_{N}\ast |\phi |^{2})\phi & & \leftarrow \text{forcing} \\
& \qquad -2[V_{N}\ast \Re (\bar{\phi}\tilde{\phi})]\phi -(V_{N}\ast |\phi
|^{2})\tilde{\phi} & & \leftarrow \text{linear in $\tilde{\phi}$} \\
& \qquad -2[V_{N}\ast \Re (\bar{\phi}\tilde{\phi})]\tilde{\phi}-V_{N}\ast |%
\tilde{\phi}|^{2})\phi & & \leftarrow \text{quadratic in $\tilde{\phi}$} \\
& \qquad -(V_{N}\ast |\tilde{\phi}|^{2})\tilde{\phi} & & \leftarrow \text{%
cubic in $\tilde{\phi}$}
\end{align*}%
By \eqref{E:UV-dual},%
\begin{eqnarray}
&&\Vert \tilde{\phi}\Vert _{U_{I,\Delta }^{2}H_{x}^{1}}  \notag \\
&\leq &\Vert \tilde{\phi}_{0}\Vert _{H_{x}^{1}}  \notag \\
&&+\sup_{\substack{ g\in V_{I,\Delta }^{2}L_{x}^{2}  \\ \Vert g\Vert
_{V_{I,\Delta }^{2}L_{x}^{2}}\leq 1}}\int_{I}\int \langle \nabla \rangle
\lbrack (W_{N}\ast |\phi |^{2})\phi ]\,g\,dx\,dt  \label{eqn:forcing} \\
&&+K\Vert \langle \nabla \rangle \left( \lbrack V_{N}\ast \Re (\bar{\phi}%
\tilde{\phi})]\phi \right) \Vert _{L_{I}^{1}L_{x}^{2}}+K\Vert \langle \nabla
\rangle \lbrack (V_{N}\ast |\phi |^{2})\tilde{\phi}]\Vert
_{L_{I}^{1}L_{x}^{2}}  \label{eqn:linear} \\
&&+K\Vert \langle \nabla \rangle \left( \lbrack V_{N}\ast \Re (\bar{\phi}%
\tilde{\phi})]\tilde{\phi}\right) \Vert _{L_{I}^{1}L_{x}^{2}}+K\Vert \langle
\nabla \rangle \lbrack (V_{N}\ast |\tilde{\phi}|^{2})\phi ]\Vert
_{L_{I}^{1}L_{x}^{2}}  \label{eqn:quadratic} \\
&&+K\Vert \langle \nabla \rangle \lbrack (V_{N}\ast |\tilde{\phi}|^{2})%
\tilde{\phi}]\Vert _{L_{I}^{1}L_{x}^{2}}  \label{eqn:cubic}
\end{eqnarray}%
in which (\ref{eqn:forcing}) corresponds to the forcing, (\ref{eqn:linear})
corresponds to the terms linear in $\tilde{\phi}$, (\ref{eqn:quadratic})
corresponds to the terms quadratic in $\tilde{\phi}$, and (\ref{eqn:cubic})
corresponds to the terms cubic in $\tilde{\phi}$, for some absolute constant 
$K>0$. The linear, quadratic, and cubic terms are estimated in a standard
way (the $V_{N}\ast $ operator is treated the same of a delta convolution (a
product pairing)), yielding a bound by 
\begin{equation}
\begin{aligned} \| \tilde \phi \|_{U^2_{I,\Delta}H_x^1} & \leq \| \tilde
\phi_0\|_{H_x^1} + Q_N(\phi) + K\|\langle \nabla \rangle \phi\|_{L_I^2
L_x^6} \| \phi \|_{L_I^\infty H_x^1} \| \langle \nabla \rangle \tilde \phi
\|_{L_I^2 L_x^6} \\ & \qquad + K\|\langle \nabla \rangle \phi\|_{L_I^2
L_x^6} \| \tilde \phi \|_{L_I^\infty H_x^1} \| \langle \nabla \rangle \tilde
\phi \|_{L_I^2 L_x^6} \\ & \qquad + K\|\langle \nabla \rangle \tilde
\phi\|_{L_I^2 L_x^6} \| \tilde \phi \|_{L_I^\infty H_x^1} \| \langle \nabla
\rangle \tilde \phi \|_{L_I^2 L_x^6} \\ &\leq \|\tilde \phi_0\|_{H_x^1} +
Q_{N,I}(\phi) +R(\tilde\phi) \|\tilde \phi\|_{U^2_{I,\Delta}H_x^1}
\end{aligned}  \label{E:UV-stuff-02}
\end{equation}%
where 
\begin{equation*}
Q_{N,I}(\phi )\overset{\mathrm{def}}{=}\sup_{\substack{ g\in V_{I,\Delta
}^{2}L_{x}^{2}  \\ \Vert g\Vert _{V_{I,\Delta }^{2}L_{x}^{2}}\leq 1}}%
\int_{I}\int \langle \nabla \rangle \lbrack (W_{N}\ast |\phi |^{2})\phi
]\,g\,dx\,dt
\end{equation*}%
\begin{equation*}
R_{I}(\tilde{\phi})\overset{\mathrm{def}}{=}KC_{1}\delta +K\delta \Vert 
\tilde{\phi}\Vert _{U_{I,\Delta }^{2}H_{x}^{1}}+K\Vert \tilde{\phi}\Vert
_{U_{I,\Delta }^{2}H_{x}^{1}}^{2}
\end{equation*}%
In fact, \eqref{E:UV-stuff-02} holds for any $I^{\prime }\subset I$. Let $%
I^{\prime }\subset I$ be the maximal subinterval on which 
\begin{equation}
\Vert \tilde{\phi}\Vert _{U_{I^{\prime },\Delta }^{2}H_{x}^{1}}\leq 4(\Vert 
\tilde{\phi}_{0}\Vert _{H_{x}^{1}}+Q_{N,I}(\phi ))  \label{E:UV-stuff-04}
\end{equation}%
(notice that it is $I^{\prime }$ on the left in $U_{I^{\prime },\Delta
}^{2}H_{x}^{1}$ but $I$ in the right in $Q_{N,I}(\phi )$) Then, provided $%
3KC_{1}\delta \leq \frac{1}{2}$, we have 
\begin{equation}
R_{I^{\prime }}(\tilde{\phi})\leq KC_{1}\delta +4K\delta (\Vert \tilde{\phi}%
_{0}\Vert _{H_{x}^{1}}+Q_{N,I}(\phi ))+16K(\Vert \phi _{0}\Vert
_{H_{x}^{1}}+Q_{N,I}(\phi ))^{2}  \label{E:UV-stuff-03}
\end{equation}%
by plugging (\ref{E:UV-stuff-04}) into $R_{I^{\prime }}(\tilde{\phi})$.
Provided $N$ is chosen sufficiently large in terms of $C_{1}$, the estimate %
\eqref{E:QNI-bound} below for $Q_{N,I}(\phi )$ will in particular imply 
\begin{equation*}
4(\Vert \tilde{\phi}_{0}\Vert _{H_{x}^{1}}+Q_{N,I}(\phi ))\leq \delta
\end{equation*}%
From this and \eqref{E:UV-stuff-03}, it follows that $R_{I^{\prime }}(\tilde{%
\phi})\leq 3KC_{1}\delta \leq \frac{1}{2}$. Substituting this into %
\eqref{E:UV-stuff-02} (with $I$ replaced by $I^{\prime }$), we obtain 
\begin{equation*}
\Vert \tilde{\phi}\Vert _{U_{I^{\prime },\Delta }^{2}H_{x}^{1}}\leq \Vert 
\tilde{\phi}_{0}\Vert _{H_{x}^{1}}+Q_{N,I}(\phi )+\frac{1}{2}\Vert \tilde{%
\phi}\Vert _{U_{I^{\prime },\Delta }^{2}H_{x}^{1}}
\end{equation*}%
or, after absorbing $\frac{1}{2}\Vert \tilde{\phi}\Vert _{U_{I^{\prime
},\Delta }^{2}H_{x}^{1}}$ into the left, 
\begin{equation*}
\Vert \tilde{\phi}\Vert _{U_{I^{\prime },\Delta }^{2}H_{x}^{1}}\leq 2\Vert 
\tilde{\phi}_{0}\Vert _{H_{x}^{1}}+2Q_{N,I}(\phi )
\end{equation*}%
This contradicts the maximality of $I^{\prime }\subset I$ satisfying %
\eqref{E:UV-stuff-04} unless $I^{\prime }=I$. Thus, we are able to conclude
that \eqref{E:UV-stuff-04} holds for $I^{\prime }=I$, which is the desired
result, once we have suitably estimated $Q_{N,I}(\phi )$.

Now we estimate $Q_{N,I}(\phi )$, which is more interesting as it uses the
sharpest available bilinear estimate \eqref{E:bil-UV02}. We use duality and
apply Lemma \ref{L:N-for-deriv-bil}. Taking $g\in V_{I,\Delta }^{2}L_{x}^{2}$
we need to show 
\begin{equation*}
\int_{I}\int \nabla \lbrack (W_{N}\ast |\phi |^{2})\;\phi
]\;g\;dx\,dt\lesssim N^{-\beta }(\log N)^{7}\Vert \phi \Vert _{U_{I,\Delta
}^{2}H_{x}^{1}}^{3}\Vert g\Vert _{V_{I,\Delta }^{2}L_{x}^{2}}
\end{equation*}%
We distribute the derivative on the left to obtain two terms 
\begin{align*}
\hspace{0.3in}& \hspace{-0.3in}\int_{I}\int \nabla \lbrack (W_{N}\ast |\phi
|^{2})\;\phi ]\;g\;dx\,dt \\
& =2\int_{I}\int [W_{N}\ast \Re (\bar{\phi}\nabla \phi )]\,\phi
\,g\,dx\,dt+\int_{I}\int [W_{N}\ast |\phi |^{2}]\,\nabla \phi \,g\,dx\,dt \\
& =\text{I}+\text{II}
\end{align*}%
For the second term, we estimate as 
\begin{equation*}
\text{II}\lesssim \Vert W_{N}\ast |\phi |^{2}\Vert
_{L_{I}^{2}L_{x}^{3}}\Vert \nabla \phi \Vert _{L_{I}^{4}L_{x}^{3}}\Vert
g\Vert _{L_{I}^{4}L_{x}^{3}}
\end{equation*}%
By Lemma \ref{L:N-for-deriv} with $s=1$ and $p=3$, 
\begin{align*}
\text{II}& \lesssim N^{-\beta }\Vert \nabla |\phi |^{2}\Vert
_{L_{I}^{2}L_{x}^{3}}\Vert \nabla \phi \Vert _{L_{I}^{4}L_{x}^{3}}\Vert
g\Vert _{L_{I}^{4}L_{x}^{3}} \\
& \lesssim N^{-\beta }\Vert \nabla \phi \Vert _{L_{I}^{2}L_{x}^{6}}\Vert
\phi \Vert _{L_{I}^{\infty }L_{x}^{6}}\Vert \nabla \phi \Vert
_{L_{I}^{4}L_{x}^{3}}\Vert g\Vert _{L_{I}^{4}L_{x}^{3}} \\
& \lesssim N^{-\beta }\Vert \phi \Vert _{U_{\Delta
,I}^{2}H_{x}^{1}}^{3}\Vert g\Vert _{V_{I,\Delta }^{2}L_{x}^{2}} \\
& \lesssim N^{-\beta }C_{1}^{3}\Vert g\Vert _{V_{I,\Delta }^{2}L_{x}^{2}}
\end{align*}%
where we have applied \eqref{E:Str-U} and also $V^{2}\hookrightarrow U^{4}$
embedding (see \eqref{E:UV-basic-embeddings}) for the $g$ term. For Term I,
however, we use the dual structure and apply Lemma \ref{L:N-for-deriv-bil}.
Applying Lemma \ref{L:N-for-deriv-bil} slightly interpolated with the
trivial estimate to insert the logarithmic terms 
\begin{equation*}
\text{I}\lesssim N^{-\beta }(\log N)^{7}\left\Vert \frac{\langle \nabla
\rangle ^{1/2}}{(\log \langle \nabla \rangle )^{3}}[\bar{\phi}\,\nabla \phi
]\right\Vert _{L_{I}^{2}L_{x}^{2}}\left\Vert \frac{\langle \nabla \rangle
^{1/2}}{(\log \langle \nabla \rangle )^{4}}[\phi g]\right\Vert
_{L_{I}^{2}L_{x}^{2}}
\end{equation*}%
We can (nearly) rescue the $\frac{1}{2}$ derivative in each term using the
bilinear Strichartz Lemma \ref{L:bil-Str} (for each $L_{I}^{2}L_{x}^{2}$
term). Employing a Littlewood-Paley decomposition 
\begin{equation*}
\begin{aligned} \text{I} \lesssim N^{-\beta}(\log N)^7
\sum_{M_1,M_2,M_3,M_4} & M_1 \frac{\max(M_1,M_2)^{1/2}}{(\log
\max(M_1,M_2))^3} \frac{\max(M_3,M_4)^{1/2}}{(\log\max (M_3,M_4))^4} \\
&\qquad \| P_{M_1} \phi \; \overline{P_{M_2}\phi} \|_{L_I^2L_x^2} \| P_{M_3}
\phi \; P_{M_4}g\|_{L_I^2L_x^2} \end{aligned}
\end{equation*}%
Applying the bilinear Strichartz estimate Lemma \ref{L:bil-Str} to the first
term, and \eqref{E:bil-UV02} to the second term, which introduces the factor 
\begin{eqnarray*}
&&\min (M_{1},M_{2})\max (M_{1},M_{2})^{-1/2}\min (M_{3},M_{4}) \\
&&\times \max (M_{3},M_{4})^{-1/2}\left( \log \frac{\max (M_{3},M_{4})}{\min
(M_{3},M_{4})}+1\right)
\end{eqnarray*}%
we obtain%
\begin{eqnarray*}
\text{I} &\lesssim &N^{-\beta }(\log N)^{7}\sum_{M_{1},M_{2},M_{3},M_{4}}(%
\frac{M_{1}\min (M_{1},M_{2})}{(\log \max (M_{1},M_{2}))^{3}} \\
&&\times \frac{\min (M_{3},M_{4})}{(\log \max (M_{3},M_{4}))^{4}}\left( \log 
\frac{\max (M_{3},M_{4})}{\min (M_{3},M_{4})}+1\right) \\
&&\times \quad \Vert P_{M_{1}}\phi \Vert _{U_{\Delta ,I}^{2}L_{x}^{2}}\Vert
P_{M_{2}}\phi \Vert _{U_{\Delta ,I}^{2}L_{x}^{2}}\Vert P_{M_{3}}\phi \Vert
_{U_{\Delta ,I}^{2}L_{x}^{2}}\Vert P_{M_{4}}g\Vert _{V_{\Delta
,I}^{2}L_{x}^{2}})
\end{eqnarray*}%
where the extra log factor comes from the need to get $V^{2}$ on the $g$
term instead of $U^{2}$, as explained above \eqref{E:bil-UV02}. Distributing
the derivatives onto each of the three $\phi $ factors,%
\begin{eqnarray*}
\text{I} &\lesssim &N^{-\beta }(\log N)^{7}\sum_{M_{1},M_{2},M_{3},M_{4}}(%
\frac{\min (M_{1},M_{2})}{M_{2}(\log \max (M_{1},M_{2}))^{3}} \\
&&\times \frac{\min (M_{3},M_{4})}{M_{3}(\log \max (M_{3},M_{4}))^{3}}\quad
\Vert P_{M_{1}}\phi \Vert _{U_{\Delta ,I}^{2}H_{x}^{1}}\Vert P_{M_{2}}\phi
\Vert _{U_{\Delta ,I}^{2}H_{x}^{1}} \\
&&\times \Vert P_{M_{3}}\phi \Vert _{U_{\Delta ,I}^{2}H_{x}^{1}}\Vert
P_{M_{4}}g\Vert _{V_{\Delta ,I}^{2}L_{x}^{2}})
\end{eqnarray*}%
Applying the estimates 
\begin{equation*}
\Vert P_{M_{j}}\phi \Vert _{U_{\Delta ,I}^{2}H_{x}^{1}}\lesssim \Vert \phi
\Vert _{U_{\Delta ,I}^{2}H_{x}^{1}}\,,\qquad \Vert P_{M_{4}}g\Vert
_{V_{\Delta ,I}^{2}L_{x}^{2}}\lesssim \Vert g\Vert _{V_{\Delta
,I}^{2}L_{x}^{2}}
\end{equation*}%
we can carry out the sum to obtain $\text{I}\lesssim N^{-\beta }(\log
N)^{7}C_{1}^{3}\Vert g\Vert _{V_{I,\Delta }^{2}L_{x}^{2}}$. Collecting the
estimates on I and II, we obtain 
\begin{equation}
Q_{N,I}(\phi )\lesssim N^{-\beta }(\log N)^{7}C_{1}^{3}  \label{E:QNI-bound}
\end{equation}
\end{proof}

\begin{proposition}
\label{P:diff-control-ep} Let $q>1$. Suppose that on a time interval $I$ the
solutions to (\ref{NLS:opening cubic}) and (\ref{NLS:Hartree}) satisfy 
\begin{equation*}
\Vert \phi \Vert _{U_{I,\Delta }^{2}H_{x}^{q}}\leq 2C_{1}\,,\qquad \Vert
\phi _{N}\Vert _{U_{I,\Delta }^{2}H_{x}^{q}}\leq 2C_{1}
\end{equation*}%
for some constant $C_{1}$ and 
\begin{equation*}
\Vert \langle \nabla \rangle^{q} \phi \Vert _{L_{I}^{2}L_{x}^{6}}\leq \delta
\,,\qquad \Vert \langle \nabla \rangle^{q} \phi _{N}\Vert
_{L_{I}^{2}L_{x}^{6}}\leq \delta
\end{equation*}%
Consider the difference 
\begin{equation*}
\tilde{\phi}(t)=\phi _{N}(t)-\phi (t)
\end{equation*}%
with initial condition $\tilde{\phi}_{0}=(\phi _{N})_{0}-\phi _{0}$ for the
time interval $I$.

Provided $\delta >0$ is chosen small, $\Vert \tilde{\phi}_{0}\Vert
_{H_{x}^{1}}$ is sufficiently small, and $N$ is sufficiently large (all of
these thresholds are expressed in terms of $C_{1}$ only), then we have 
\begin{equation}  \label{E:diff-control-ep}
\Vert \tilde{\phi}\Vert _{U_{I,\Delta }^{2}H_{x}^{1}}\lesssim \Vert \tilde{%
\phi}_{0}\Vert _{H_{x}^{1}}+N^{-q\beta }(\log N)^7 C_{1}^{3}
\end{equation}
\end{proposition}

\begin{proof}
The proof follows that of Proposition \ref{P:diff-control}, with the only
modification needed in the treatment of the forcing term 
\begin{equation*}
Q_{N,I}(\phi )\overset{\mathrm{def}}{=}\sup_{\substack{ g\in V_{I,\Delta
}^{2}L_{x}^{2}  \\ \Vert g\Vert _{V_{I,\Delta }^{2}L_{x}^{2}}\leq 1}}%
\int_{I}\int \langle \nabla \rangle \lbrack (W_{N}\ast |\phi |^{2})\phi
]\,g\,dx\,dt
\end{equation*}%
After distributing the derivative on the left we obtain two terms 
\begin{align*}
\hspace{0.3in}& \hspace{-0.3in}\int_{I}\int \nabla \lbrack (W_{N}\ast |\phi
|^{2})\;\phi ]\;g\;dx\,dt \\
& =2\int_{I}\int [W_{N}\ast \Re (\bar{\phi}\nabla \phi )]\,\phi
\,g\,dx\,dt+\int_{I}\int [W_{N}\ast |\phi |^{2}]\,\nabla \phi \,g\,dx\,dt \\
& =\text{I}+\text{II}
\end{align*}%
Term II is estimated as in the proof of Proposition \ref{P:diff-control},
giving the bound 
\begin{equation*}
\text{II}\lesssim N^{-\min (q,2)\beta }C_{1}^{3}\Vert g\Vert _{V_{I,\Delta
}^{2}L_{x}^{2}}
\end{equation*}%
Term I is also estimated as in the proof of Proposition \ref{P:diff-control}%
. Moreover, we apply Lemma \ref{L:N-for-deriv} on the left product $%
W_{N}\ast (\phi \;\nabla \phi )$ with $s=q-\frac{1}{2}$, and apply lemma \ref%
{L:N-for-deriv} on the right product $W_{N}\ast (\phi g)$ with $s=\frac{1}{2}
$ to obtain (dropping log factors for clarity)%
\begin{eqnarray*}
\text{I} &\lesssim &(\log N)^{7}N^{-q\beta
}\sum_{M_{1},M_{2},M_{3},M_{4}}(\max (M_{1},M_{2})^{q}\min (M_{1},M_{2}) \\
&&\times \min (M_{3},M_{4})\Vert P_{M_{1}}\phi \Vert _{U_{\Delta
,I}^{2}L_{x}^{2}}\Vert P_{M_{2}}\phi \Vert _{U_{\Delta
,I}^{2}L_{x}^{2}}\Vert P_{M_{3}}\phi \Vert _{U_{\Delta
,I}^{2}L_{x}^{2}}\Vert P_{M_{4}}g\Vert _{V_{\Delta ,I}^{2}L_{x}^{2}})
\end{eqnarray*}
Putting $q$ derivatives onto each $\phi $ factor gives%
\begin{eqnarray*}
\text{I} &\lesssim &(\log N)^{7}N^{-q\beta
}\sum_{M_{1},M_{2},M_{3},M_{4}}\min (M_{1},M_{2})^{-(q-1)}\min
(M_{3},M_{4})M_{3}^{-q} \\
&&\times \Vert P_{M_{1}}\phi \Vert _{U_{\Delta ,I}^{2}H_{x}^{q}}\Vert
P_{M_{2}}\phi \Vert _{U_{\Delta ,I}^{2}H_{x}^{q}}\Vert P_{M_{3}}\phi \Vert
_{U_{\Delta ,I}^{2}H_{x}^{q}}\Vert P_{M_{4}}g\Vert _{V_{\Delta
,I}^{2}L_{x}^{2}}
\end{eqnarray*}
Now carry out the sum (recall that the introduction of $(\log N)^{7}$
provided $(\log M_{i})$ factors that allow us to sum).
\end{proof}

We can now conclude the proof of Theorem \ref{Thm:biScattering}. Recall
there is a finite partition of $0 \leq t<+\infty$ 
\begin{equation*}
0 = t_0 < t_1< \cdots < t_J= +\infty
\end{equation*}
such that for each $I=[t_{j-1},t_j)$, the solutions $\phi$ and $\phi_M$ are
small in the Strichartz norms, i.e. \eqref{E:Strichartz-small} holds. By
Corollary \ref{C:U2control}, the $U^2$ norms of $\phi$ and $\phi_N$ are
controlled, i.e. \eqref{E:U2control} holds. Thus, for each time interval $%
I=[t_{j-1},t_j)$, the hypotheses of Proposition \ref{P:diff-control} are
satisfied, and \eqref{E:diff-control} holds. This implies, in particular,
that 
\begin{equation*}
\| \tilde \phi(t) \|_{L_{[t_{j-1},t_j]}^\infty H_x^1} \lesssim \|\tilde
\phi(t_{j-1}) \|_{H_x^1} + C N^{-\beta} (\log N)^7
\end{equation*}
Therefore, the estimate on the $j$th interval feeds into the estimate for
the $(j+1)$st interval, and since there are only a finite number of time
intervals, we can reach all time.

For the $H^{q}$ version, we apply a persistence of regularity argument, as
in Bourgain \cite{B1998}, to deduce that the $H^{q}$ norms of $\phi $ and $%
\phi _{N}$ are globally bounded, and that modifications to $q$ regularity of
Theorem \ref{T:scat} and Corollary \ref{C:U2control} follow. Thus, on each
time interval $I=[t_{j-1},t_{j})$, the hypotheses of Proposition \ref%
{P:diff-control-ep} are satisfied, and \eqref{E:diff-control-ep} holds. This
implies, in particular, that 
\begin{equation*}
\Vert \tilde{\phi}(t)\Vert _{L_{[t_{j-1},t_{j}]}^{\infty }H_{x}^{1}}\lesssim
\Vert \tilde{\phi}(t_{j-1})\Vert _{H_{x}^{1}}+CN^{-q\beta }(\log N)^{7}
\end{equation*}%
Therefore, the estimate on the $j$th interval feeds into the estimate for
the $(j+1)$st interval, and since there are only a finite number of time
intervals, we can reach all time.

Next we address the proof of the difference estimate 
\eqref{estimate:key-2
for almost optimal}. These estimates will follow from the lemma below since 
\begin{equation*}
\sum_{k=0}^{+\infty }Z^{-k}(2k)(3C_{1})^{2k-1}<\infty
\end{equation*}%
provided $Z>(3C_{1})^{2}$.

\begin{lemma}
\label{Lemma:Binomial}Let 
\begin{equation*}
G_{N,k}=\phi _{N}(x_{1})\cdots \phi _{N}(x_{k})\overline{\phi _{N}}%
(x_{1}^{\prime })\cdots \overline{\phi _{N}}(x_{k}^{\prime })
\end{equation*}%
\begin{equation*}
G_{k}=\phi (x_{1})\cdots \phi (x_{k})\overline{\phi }(x_{1}^{\prime })\cdots 
\overline{\phi }(x_{k}^{\prime })
\end{equation*}%
and let $\tilde{\phi}=\phi _{N}-\phi $. If 
\begin{equation*}
\Vert \phi _{N}\Vert _{H^{1}}\leq C_{1}\quad \text{and}\quad \Vert \phi
\Vert _{H^{1}}\leq C_{1}
\end{equation*}%
then 
\begin{equation*}
\Vert G_{N,k}-G_{k}\Vert _{H_{\mathbf{x}_{k},\mathbf{x}_{k}^{\prime
}}^{1}}\leq 2k(3C_{1})^{2k-1}\Vert \tilde{\phi}\Vert _{H^{1}}
\end{equation*}
\end{lemma}

\begin{proof}
In the formula for $G_{N,k}$, replace each instance of $\phi _{N}$ by $%
\tilde{\phi}+\phi $, and expand to a sum of $2^{2k}$ terms, and note that
passing to the difference $G_{N,k}-G_{k}$ removes one of these terms. Apply
the $H_{\mathbf{x}_{k},\mathbf{x}_{k}^{\prime }}^{1}$ norm, bound via
Minkowski's inequality by a sum with the norm on each of the individual
terms. At this point, the $2^{2k}-1$ terms can be grouped into $2k$ terms: 
\begin{equation*}
\Vert G_{N,k}-G_{k}\Vert _{H_{\mathbf{x}_{k},\mathbf{x}_{k}^{\prime
}}^{1}}\leq \sum_{\ell =1}^{2k}\binom{2k}{\ell }\Vert \tilde{\phi}\Vert
_{H^{1}}^{\ell }\Vert \phi \Vert _{H^{1}}^{2k-\ell }
\end{equation*}%
where we note that the sum starts at $\ell =1$ and not $\ell =0$. Applying
the bound $\binom{2k}{\ell }\leq 2k\binom{2k-1}{\ell -1}$ and reindexing the
sum with $j=\ell +1$, 
\begin{eqnarray*}
&&\Vert G_{N,k}-G_{k}\Vert _{H_{\mathbf{x}_{k},\mathbf{x}_{k}^{\prime }}^{1}}
\\
&\leq &2k\Vert \tilde{\phi}\Vert _{H^{1}}\sum_{j=0}^{2k-1}\binom{2k-1}{j}%
\Vert \tilde{\phi}\Vert _{H^{1}}^{j}\Vert \phi \Vert _{H^{1}}^{2k-1-j} \\
&=&2k\Vert \tilde{\phi}\Vert _{H^{1}}(\Vert \tilde{\phi}\Vert _{H^{1}}+\Vert
\phi \Vert _{H^{1}})^{2k-1}
\end{eqnarray*}%
From this, the claimed estimate follows.
\end{proof}

\subsection{Optimality via Space-time Resonance\label{sec:optimality}}

We can in fact provide an example showing that $N^{-\beta q}$ is optimal for 
$\phi \in H^q$, $q\geq 1$. Consider the main forcing term in the equation
for $\tilde{\phi}$ 
\begin{equation}  \label{E:F-def}
F(t) = \int_0^t e^{i(t-t^{\prime })\Delta} [(W_N*|\phi|^2)\, \phi] \,
dt^{\prime }
\end{equation}
where $\phi \in H^q$ is a scattering solution to NLS. For simplicity, let us
replace $\phi$ by a linear solution, i.e. take 
\begin{equation*}
\phi(t) = e^{it\Delta} f
\end{equation*}
for some $f=f(x)$, which is a natural benchmark on which to assess $F(t)$ in %
\eqref{E:F-def} since the NLS solution $\phi$ scatters.

\begin{lemma}
For $q\geq 1$, and for each $N \gg 1$, there exists a choice of $f$ for
which $\|f\|_{H^q} =1$ and $F(t)$ given by \eqref{E:F-def} with $\phi(t) =
e^{it\Delta}f$ satisfies 
\begin{equation*}
\|F(t)\|_{L_{t\in [0,1]}^\infty H_x^1} \gtrsim N^{-q\beta}
\end{equation*}
\end{lemma}

\begin{proof}
The strategy is to concoct a choice for $f$ in which the frequency support
is sufficiently constrained so as to produce a resonant interaction.

Denote spatial coordinates by $x=(x_{1},x_{2},x_{3})$ and frequency
coordinates by $\xi =(\xi _{1},\xi _{2},\xi _{3})$. Consider the following
choice for $f$: 
\begin{equation*}
\hat{f}(\xi )=\Big(N^{\beta /2}\mathbf{1}_{(0,N^{-\beta })}(\xi
_{1})+N^{-\beta(q-\frac12)}\mathbf{1}_{(N^{\beta },N^{\beta }+N^{-\beta
})}(\xi _{1})\Big)\mathbf{1}_{(0,1)}(\xi _{2})\mathbf{1}_{(0,1)}(\xi _{3})
\end{equation*}%
so that $\Vert f\Vert _{H^{q}}\sim O(1)$. The reason for choosing intervals
of width $N^{-\beta }$ is that $N^{2\beta }\leq (N^{\beta }+\delta N^{-\beta
})^{2}\leq N^{2\beta }+O(1)$ for $0\leq \delta \leq 1$ -- that is, we have $%
O(1)$ resolution of the square. With this choice, 
\begin{align*}
\hspace{0.3in}& \hspace{-0.3in}(\hat{\phi}(t^{\prime })\ast \hat{\bar{\phi}}%
(t^{\prime }))(\eta ) =\int \hat{\phi}(\eta -\rho ,t^{\prime })\hat{\bar{\phi%
}}(\eta ,t^{\prime })\,d\rho \\
& =\int_{\rho }e^{-it^{\prime}(\eta-\rho)^2}e^{it^{\prime}\rho^2}\hat{f}%
(\eta -\rho )\hat{f}(\rho )\,d\rho \\
& \sim e^{it^{\prime }\cdot O(1)}\Big(\mathbf{1}_{(0,N^{-\beta })}(\eta
_{1})+N^{-q\beta }(e^{-it^{\prime} N^{2\beta} }+e^{it^{\prime}N^{2\beta}})%
\mathbf{1}_{(N^{\beta },N^{\beta }+N^{-\beta })}(\eta _{1}) \\
& \qquad \qquad +N^{-2q\beta }\mathbf{1}_{(2N^{\beta },2N^{\beta }+N^{-\beta
})}(\eta _{1})\Big)\mathbf{1}_{(0,1)}(\eta _{2})\mathbf{1}_{(0,1)}(\eta _{3})
\end{align*}%
Applying $W_{N}*$ basically removes the first term $\mathbf{1}_{(0,N^{-\beta
})}(\eta _{1})$, 
\begin{align*}
\hspace{0.3in}& \hspace{-0.3in} [W_N * |\phi(t^{\prime})|^2]\widehat{\;}%
(\eta ) \\
& \sim e^{it^{\prime }\cdot O(1)}\Big(N^{-q\beta }(e^{-it^{\prime}
N^{2\beta} }+e^{it^{\prime}N^{2\beta}})\mathbf{1}_{(N^{\beta },N^{\beta
}+N^{-\beta })}(\eta _{1}) \\
& \qquad \qquad +N^{-2q\beta }\mathbf{1}_{(2N^{\beta },2N^{\beta }+N^{-\beta
})}(\eta _{1})\Big)\mathbf{1}_{(0,1)}(\eta _{2})\mathbf{1}_{(0,1)}(\eta _{3})
\end{align*}%
Now consider 
\begin{equation}  \label{E:opt-01}
[ (W_N*|\phi(t^{\prime})|^2) \phi(t^{\prime}) ]\widehat{\;}(\xi)= \int_\eta
(W_N*|\phi(t^{\prime})|^2)\widehat{\;}(\eta) \widehat{\phi(t^{\prime})}%
(\xi-\eta) \, d\eta
\end{equation}
We have $\widehat{\phi (t^{\prime })}(\xi -\eta
)=e^{-it^{\prime}(\xi-\eta)^2 }\hat{f}(\xi -\eta )$. In the product inside
the integrand in \eqref{E:opt-01}, we have either 
\begin{equation*}
N^{\beta }\leq \eta _{1}\leq N^{\beta }+N^{-\beta }\,,\quad \text{or}\quad
N^{2\beta }\leq \eta _{1}\leq N^{2\beta }+N^{-\beta }
\end{equation*}%
and 
\begin{equation*}
0\leq \xi _{1}-\eta _{1}\leq N^{-\beta }\,,\quad \text{or}\quad N^{\beta
}\leq \xi _{1}-\eta _{1}\leq N^{\beta }+N^{-\beta }
\end{equation*}%
Thus, in the product, there are four terms resulting from all possible cross
pairings.

The dominant term of interest for us (that produces the lower bound) will
arise from the case when $N^{\beta }\leq \eta _{1}\leq N^{\beta }+N^{-\beta
} $ pairs with $0\leq \xi _{1}-\eta _{1}\leq N^{-\beta }$, (so that the
output frequency $\xi $ satisfies $N^{\beta }\leq \xi _{1}\leq N^{\beta
}+N^{-\beta }$). The result is 
\begin{eqnarray}
&&\Big(\lbrack W_{N}\ast |\phi (t^{\prime })|^{2}]\phi (t^{\prime })\Big)%
\widehat{\;}(\xi )  \label{E:opt-02} \\
&\sim &e^{it^{\prime }\cdot O(1)}N^{-\beta (q+\frac{1}{2})}(e^{-it^{\prime
}N^{2\beta }}+e^{it^{\prime }N^{2\beta }})\mathbf{1}_{(N^{\beta },N^{\beta
}+N^{-\beta })}(\xi _{1})\mathbf{1}_{(0,1)}(\xi _{2})\mathbf{1}_{(0,1)}(\xi
_{3})  \notag
\end{eqnarray}%
where we have dropped the terms that will have subordinate effect. The
coefficient $N^{-\beta (q+\frac{1}{2})}$ comes from the product of three
things: the $N^{-q\beta }$ coefficient in $[W_{N}\ast |\phi (t^{\prime
})|^{2}]\widehat{\;}(\eta )$, the $N^{\beta /2}$ coefficient in $\widehat{%
\phi (t^{\prime })}(\xi -\eta )$, and the size of the $\eta $-integration
support, which is $N^{-\beta }$. Finally we come to 
\begin{eqnarray*}
&&\left[ \int_{0}^{t}e^{i(t-t^{\prime })\Delta }[W_{N}\ast |\phi (t^{\prime
})|^{2}]\phi (t^{\prime })\,dt^{\prime }\right] \widehat{\;}(\xi ) \\
&=&e^{-it\xi ^{2}}\int_{0}^{t}e^{it^{\prime }\xi ^{2}}\Big(\lbrack
W_{N}|\phi (t^{\prime 2}]\phi (t^{\prime })\Big)\widehat{\;}(\xi
)\,dt^{\prime }
\end{eqnarray*}%
Plugging in the term above, it is key to notice that in the phase product
inside the integrand 
\begin{equation*}
e^{-it^{\prime }\xi ^{2}}(e^{-it^{\prime }N^{2\beta }}+e^{it^{\prime
}N^{2\beta }})=(e^{-2it^{\prime }N^{2\beta }}+1)
\end{equation*}%
so there is a non-oscillatory (resonant) component. Thus, when the time
integral is carried out, this term survives, and gives 
\begin{equation*}
\sim N^{-\beta (q+\frac{1}{2})}\mathbf{1}_{(N^{\beta },N^{\beta }+N^{-\beta
})}(\xi _{1})\mathbf{1}_{(0,1)}(\xi _{2})\mathbf{1}_{(0,1)}(\xi _{3})
\end{equation*}%
The $H^{1}$ norm of this term is $\sim N^{-q\beta }$. As other terms are
subordinate, this term contributing $N^{-q\beta }$ to the norm becomes a
lower bound.
\end{proof}

\appendix

\section{Misc. {Estimates}}

\subsection{Collapsing Estimates and Strichartz Estimates}

We use the original Klainerman-Machedon collapsing estimate as our iterating
estimate in this paper.

\begin{lemma}[\protect\cite{KM,TCNP5,Chen3DDerivation}]
\label{Lemma:H^1K-M Estimate}\footnote{%
For more estimates of this type, see \cite%
{ChenDie,ChenAnisotropic,GSS,GM,KSS}.}There is a $C$ independent of $V,j,k$,
and $N$ such that, (for $f^{(k+1)}(\mathbf{x}_{k+1},\mathbf{x}_{k+1}^{\prime
})$ independent of $t$) 
\begin{equation*}
\left\Vert S^{(1,k)}B_{N,j,k+1}U^{(k+1)}(t)f^{(k+1)}\right\Vert
_{L_{t}^{2}L_{\mathbf{x},\mathbf{x}^{\prime }}^{2}}\leqslant C\left\Vert
V\right\Vert _{L^{1}}\left\Vert S^{(1,k+1)}f^{(k+1)}\right\Vert _{L_{\mathbf{%
x},\mathbf{x}^{\prime }}^{2}}.
\end{equation*}
\end{lemma}

To explore the time derivative gain by Duhamel type terms, we also need the $%
X_{s,b}$ version of Lemma \ref{Lemma:H^1K-M Estimate}. As we are using $%
S^{(k)}$ to denote the space derivatives, we surpress the $s$ notation in
definition of the $X_{s,b}$ space and define the norm $X_{b}^{(k)}$ by 
\begin{equation*}
\Vert \alpha ^{(k)}\Vert _{X_{b}^{(k)}}=\left( \int \langle \tau +\left\vert 
\mathbf{\xi }_{k}\right\vert ^{2}-\left\vert \mathbf{\xi }_{k}^{\prime
}\right\vert ^{2}\rangle ^{2b}\left\vert \hat{\alpha}^{(k)}(\tau ,\mathbf{%
\xi }_{k},\mathbf{\xi }_{k}^{\prime })\right\vert ^{2}\,d\tau \,d\mathbf{\xi 
}_{k}\,d\mathbf{\xi }_{k}^{\prime }\right) ^{1/2}
\end{equation*}%
which is essentially a $X_{0,b}$ norm. We then have the Duhamel
time-derivative gain property and the $X_{s,b}$ version of Lemma \ref%
{Lemma:H^1K-M Estimate}.

\begin{claim}[\protect\cite{C-H2/3}]
\label{Claim:b to b-1}Let $\frac{1}{2}<b<1$ and $\theta (t)$ be a smooth
cutoff. Then 
\begin{equation}
\left\Vert \theta (t)\int_{0}^{t}U^{(k)}(t-s)\beta ^{(k)}(s)\,ds\right\Vert
_{X_{b}^{(k)}}\lesssim \Vert \beta ^{(k)}\Vert _{X_{b-1}^{(k)}}
\label{E:X-1}
\end{equation}
\end{claim}

\begin{lemma}[\protect\cite{C-H2/3}]
\label{Lemma:KMEstimateInWithX_b}There is a $C$ independent of $j,k$, and $N$
such that (for $\alpha ^{(k+1)}(t,\mathbf{x}_{k+1},\mathbf{x}_{k+1})$
dependent on $t$) 
\begin{equation*}
\Vert S^{(1,k)}B_{N,j,k+1}\alpha ^{(k+1)}\Vert _{L_{t}^{2}L_{\mathbf{x},%
\mathbf{x}^{\prime }}^{2}}\leqslant C\Vert S^{(1,k+1)}\alpha ^{(k+1)}\Vert
_{X_{\frac{1}{2}+}^{(k+1)}}
\end{equation*}
\end{lemma}

In the above notation, the dual Strichartz estimates we need in this paper
are the following:

\begin{lemma}[\protect\cite{C-H2/3}]
\label{Lemma:PP Estimate in Strichartz form}Let 
\begin{equation*}
\beta ^{(k)}(t,\mathbf{x}_{k},\mathbf{x}_{k}^{\prime })=N^{3\beta
-1}V(N^{\beta }(x_{i}-x_{j}))\gamma ^{(k)}(t,\mathbf{x}_{k},\mathbf{x}%
_{k}^{\prime })
\end{equation*}%
Then for $N\geq 1$, we have 
\begin{equation}
\Vert \left\vert \nabla _{x_{i}}\right\vert \left\vert \nabla
_{x_{j}}\right\vert \beta ^{(k)}\Vert _{X_{-\frac{1}{2}+}^{(k)}}\lesssim N^{%
\frac{5}{2}\beta -1}\Vert \langle \nabla _{x_{i}}\rangle \langle \nabla
_{x_{j}}\rangle \gamma ^{(k)}\Vert _{L_{t}^{2}L_{\mathbf{x}\mathbf{x}%
^{\prime }}^{2}}  \label{E:StrCor1}
\end{equation}%
and 
\begin{equation}
\Vert \beta ^{(k)}\Vert _{X_{-\frac{1}{2}+}^{(k)}}\lesssim N^{\frac{1}{2}%
\beta -1}\Vert \langle \nabla _{x_{i}}\rangle \langle \nabla _{x_{j}}\rangle
\gamma ^{(k)}\Vert _{L_{t}^{2}L_{\mathbf{x}\mathbf{x}^{\prime }}^{2}}.
\label{E:StrCor2}
\end{equation}
\end{lemma}

\subsection{Convolution Estimates}

\begin{lemma}
\label{L:N-for-deriv} Let $W_{N}(x)=N^{3\beta }V(N^{\beta }x)-b_{0}\delta
(x) $, where $b_{0}=\int V(x)\,dx$. For any $0\leq s\leq 1$, 
\begin{equation*}
\Vert W_{N}\ast f\Vert _{L_{x}^{p}}\lesssim N^{-\beta s}\Vert D^{s}f\Vert
_{L_{x}^{p}}
\end{equation*}%
for any $1<p<\infty $. The implicit constant depends only on $\Vert \langle
x\rangle V(x)\Vert _{L^{1}}$.
\end{lemma}

\begin{proof}
The case $s=0$ is just Young's inequality, since $\Vert V_{N}\Vert
_{L^{1}}=\Vert V\Vert _{L^{1}}<\infty $, independent of $N$. We next
establish the estimate for $s=1$. Since $\hat{V}(0)=b_{0}$, 
\begin{equation*}
\widehat{W_{N}}(\xi )=\hat{V}(\xi N^{-\beta })-b_{0}=\int_{s=0}^{s=1}\frac{d%
}{ds}\hat{V}(s\xi N^{-\beta })\,ds=\int_{s=0}^{s=1}N^{-\beta }\xi \cdot
\nabla \hat{V}(s\xi N^{-\beta })\,ds
\end{equation*}%
and thus 
\begin{equation*}
\widehat{W_{N}}(\xi )\hat{f}(\xi )=N^{-\beta }\int_{s=0}^{s=1}\nabla \hat{V}%
(s\xi N^{-\beta })\cdot \xi \hat{f}(\xi )\,ds
\end{equation*}%
Let $Y(x)=xV(x)$ so that $\hat{Y}(\xi )=\nabla \hat{V}(\xi )$. It follows
that 
\begin{eqnarray*}
&&\int_{y\in \mathbb{R}^{3}}W_{N}(x-y)f(y)\,dy \\
&=&N^{-\beta }\int_{s=0}^{s=1}\int_{y\in \mathbb{R}^{3}}s^{-3}N^{3\beta
}Y(s^{-1}N^{\beta }(x-y))\,\nabla f(y)\,dy\,ds
\end{eqnarray*}%
By Minkowski's inequality and Young's inequality, 
\begin{align*}
\hspace{10pt}& \hspace{-10pt}\left\Vert \int_{y\in \mathbb{R}%
^{3}}W_{N}(x-y)f(y)\,dy\right\Vert _{L_{x}^{p}} \\
& \lesssim N^{-\beta }\int_{s=0}^{1}\left\Vert \int_{y\in \mathbb{R}%
^{3}}s^{-3}N^{3\beta }Y(s^{-1}N^{\beta }(x-y))\,\nabla f(y)\,dy\right\Vert
_{L_{x}^{p}}\,ds \\
& \lesssim N^{-\beta }\Vert \nabla f\Vert _{L_{x}^{p}}
\end{align*}%
The cases $0<s<1$ follow by interpolation, as follows. Let $P_{M}$ be the
Littlewood-Paley projector for frequency $0<M<\infty $. Then by the $s=0$
and $s=1$ cases, 
\begin{align*}
\hspace{10pt}& \hspace{-10pt}\Vert W_{N}\ast f\Vert _{L_{x}^{p}}=\left\Vert
W_{N}\ast \sum_{M}P_{M}f\right\Vert _{L_{x}^{p}}\leq \sum_{M}\Vert W_{N}\ast
P_{M}f\Vert _{L_{x}^{p}} \\
& \lesssim \sum_{M}\min (1,N^{-\beta }M)\Vert P_{M}f\Vert
_{L_{x}^{p}}\lesssim \sum_{M}\min (1,N^{-\beta }M)M^{-s}\Vert D^{s}f\Vert
_{L_{x}^{p}}
\end{align*}%
Divide the sum into the case $M\leq N^{\beta }$, for which we use $\min
(1,N^{-\beta }M)=N^{-\beta }M$, and the case $M\geq N^{\beta }$, for which
we use $\min (1,N^{-\beta }M)=1$. 
\begin{equation*}
\lesssim \left( \sum_{M\leq N^{\beta }}N^{-\beta }M^{1-s}+\sum_{M\geq
N^{\beta }}M^{-s}\right) \Vert D^{s}f\Vert _{L_{x}^{p}}\lesssim N^{-\beta
s}\Vert D^{s}f\Vert _{L_{x}^{p}}
\end{equation*}
\end{proof}

\begin{lemma}
\label{L:N-for-deriv-bil}Let $W_{N}(x)=N^{3\beta }V(N^{\beta }x)-b_{0}\delta
(x)$, where $b_{0}=\int V(x)\,dx$. 
\begin{equation*}
\int (W_{N}\ast f_{1})f_{2}\,dx\lesssim N^{-\beta }\Vert |\nabla
|^{1/2}f_{1}\Vert _{L_{x}^{2}}\Vert |\nabla |^{1/2}f_{2}\Vert _{L_{x}^{2}}
\end{equation*}%
Also, if $f_{j}$ is replaced by $P_{M_{j}}f_{j}$, then the same estimate
holds but in addition we must have $M_{1}\sim M_{2}$ (or otherwise the left
side is zero).
\end{lemma}

\begin{proof}
By Plancherel 
\begin{equation}
\int (W_{N}\ast f_{1})f_{2}\,dx=\int_{\xi }\hat{W}_{N}(\xi )f_{1}(\xi
)f_{2}(\xi )\,d\xi  \label{E:Nfd1}
\end{equation}%
As in the proof of Lemma \ref{L:N-for-deriv}, 
\begin{equation*}
\widehat{W_{N}}(\xi )=N^{-\beta }\hat{Q}_{N}(\xi )\cdot \xi \,,\qquad \hat{Q}%
_{N}(\xi )\overset{\mathrm{def}}{=}\int_{s=0}^{s=1}\nabla \hat{V}(s\xi
N^{-\beta })\,ds
\end{equation*}%
Since 
\begin{equation*}
\Vert \hat{Q}_{N}\Vert _{L_{\xi }^{\infty }}\leq \Vert \nabla \hat{V}\Vert
_{L_{\xi }^{\infty }}=\Vert \lbrack xV(x)]\widehat{\;\,}\Vert _{L_{\xi
}^{\infty }}\leq \Vert xV(x)\Vert _{L_{x}^{1}}
\end{equation*}%
we can just complete the proof by Cauchy-Schwarz in \eqref{E:Nfd1}
\end{proof}

\paragraph{Acknowledgments}

The authors would like to thank Shunlin Shen and the referees for their
careful reading and checking of the paper. X.C. was partially supported by
the NSF grant DMS-2005469 and by a Simons Fellowship. J.H. was supported in
part by the NSF grant DMS-2055072.

\end{document}